%% file: ex_article.tex
\documentclass[hidelinks,onefignum,onetabnum]{siamart251216}

\input{ex_shared}

\ifpdf
\hypersetup{
  pdftitle={The Influence of the Cluster Point on Rounding Errors and Sensitivity in the Spectral Limited-Memory Preconditioner},
  pdfauthor={H Elzayyadi and J. M. Tabeart}
}
\fi

\begin{document}

\maketitle

\begin{abstract}
The spectral limited-memory preconditioner (sLMP) clusters leading eigenvalues of symmetric positive definite matrices to accelerate conjugate gradient (CG) convergence. In practice, the cluster point is often chosen to be unity. In some cases, however, this choice can fail to accelerate convergence relative to unpreconditioned CG, even when highly accurate spectral information is available. Alternative cluster points have been proposed based on exact-arithmetic convergence analysis, but such analysis does not explain this finite-precision behaviour. We study how the cluster point influences two sources of numerical error in sLMP-preconditioned CG. First, we analyse the propagation of floating-point rounding errors during application of the preconditioner and derive computable relative-error bounds. For the dominant subspace (spanned by the eigenvectors associated with the leading eigenvalues of the unpreconditioned system) and its orthogonal complement (spanned by the remaining eigenvectors), these bounds are minimized by a weighted median and a weighted arithmetic mean of the leading eigenvalues, respectively. Our analysis explains why small cluster points can strongly amplify errors in the dominant subspace. Second, we investigate sensitivity to perturbations in the dominant spectral information when constructing the preconditioner. The resulting perturbation bound is minimized by a weighted median of the perturbed dominant eigenvalues, with weights determined by the eigenvector perturbation magnitudes. Numerical experiments on synthetic problems illustrate the predicted rounding-error and sensitivity behaviour. Together, these results show that cluster-point selection in finite precision should account for exact-arithmetic convergence, rounding errors, and inaccuracies in the available spectral information.
\end{abstract}

\begin{keywords}
spectral limited-memory preconditioner, rounding errors, sensitivity analysis, spectral perturbations, conjugate gradient method
\end{keywords}

\begin{MSCcodes}
65F08, 65G50, 65F15
\end{MSCcodes}

\section{Introduction}
Large-scale SPD linear systems arise throughout scientific computing and are commonly solved using iterative methods~\cite{saad2003iterative}. One important application is weak-constraint four-dimensional variational data assimilation, where Gauss–Newton methods require the repeated solution of large-scale linear systems~\cite{freitag2020numerical,gratton2007approximate}. In practical settings, computational constraints may require the iterative solver to be terminated after a limited number of iterations, making rapid convergence in the early iterations particularly important~\cite{lawless2006inner, dauvzickaite2025introduction}. For problems of this type, the Conjugate Gradient (CG) method is one of the most widely used iterative solvers due to its favorable convergence properties in exact arithmetic~\cite{hestenes1952methods}. The convergence of CG is strongly influenced by the spectral distribution of the system matrix~\cite{carson2024towards}. Consequently, preconditioning techniques are commonly employed to improve the spectral properties of the system by clustering eigenvalues and reducing the effective condition number, thereby accelerating convergence~\cite{saad2003iterative}.

Among the many preconditioning techniques proposed for SPD systems, sLMP has proven effective in accelerating CG convergence, particularly in the context of variational data assimilation for numerical weather prediction~\cite{gratton2011class,tshimanga2008limited,dauvzickaite2021randomised}. The preconditioner is constructed using the $k$ dominant eigenpairs of the coefficient matrix, corresponding to its largest eigenvalues, and modifies the dominant part of the spectrum by clustering these eigenvalues around a prescribed cluster point $\theta$. The performance of the preconditioner therefore depends on the choice of $\theta$. In practice, $\theta=1$ is often adopted due to its simplicity and compatibility with existing preconditioning frameworks~\cite{gratton2011class,tshimanga2008limited}. More recently, Diouane et al.~\cite{diouane2024efficient} investigated the influence of the cluster point $\theta$ on the convergence of preconditioned CG and showed that the spectral distribution of the preconditioned system can be significantly improved by appropriate choice of $\theta$. Their analysis is conducted entirely in exact arithmetic and is motivated by spectral clustering considerations. In particular, they seek cluster points that improve convergence in the early iterations by modifying the spectrum of the preconditioned operator.

While this provides valuable insight into the convergence properties of sLMP, it does not fully explain an important practical observation: the standard choice $\theta=1$ may exhibit worse convergence than the unpreconditioned CG~\cite{diouane2024efficient}. As demonstrated in Subsection~\ref{Rounding error Analysis}, this deterioration may persist even when highly accurate spectral information is available. More generally, a choice of $\theta$ that is attractive from an exact-arithmetic convergence perspective may simultaneously lead to poorer numerical behaviour in practical computations. This raises a fundamental question: how should the cluster point be chosen when accounting for both rounding errors in the application of sLMP and inaccuracies in the spectral information used to construct it? More generally, it is well established that finite-precision arithmetic can alter the practical behaviour of Krylov methods, making exact-arithmetic analysis alone insufficient to predict numerical performance~\cite{meurant2006lanczos}.

In this work, we consider two distinct sources of error: floating-point rounding errors generated during the application of sLMP and inaccuracies in the spectral information used to construct the preconditioner. We analyse how these effects depend on the choice of the cluster point. We first study how the choice of $\theta$ influences the propagation of rounding errors during the application of sLMP. We then investigate the sensitivity of the preconditioner to inaccuracies in the spectral information used to construct it and their impact on the convergence of preconditioned CG. For each source of error, we derive bounds that quantify its dependence on the cluster point and identify the cluster point values that minimize the corresponding bounds. Together, these results complement existing exact-arithmetic convergence analyses by accounting for rounding errors in the application of sLMP and sensitivity to inaccuracies in the spectral information.

The remainder of this paper is organized as follows. Section~\ref{sec: Background} reviews sLMP
and existing exact-arithmetic cluster-point choices. Section~\ref{sec: Analysis} presents the
rounding-error and sensitivity analyses, Section~\ref{sec: results} examines their implications
numerically, and we present our conclusions in Section~\ref{sec: conclusion}.
\section{Background}
\label{sec: Background}
Let $A \in \mathbb{R}^{n \times n}$ be SPD. The limited-memory preconditioner (LMP)~\cite{gratton2011class} is defined as
\begin{equation}
H_k = [I_n - S(S^TAS)^{-1}S^TA][I_n - AS(S^TAS)^{-1}S^T] + \theta S(S^TAS)^{-1}S^T,
\label{LMP}
\end{equation}
where $S \in \mathbb{R}^{n \times k}$ has full column rank, $k << n$, and $\theta > 0$ is a scalar parameter.

In the spectral variant, $S$ contains the eigenvectors associated with the $k$ largest eigenvalues of $A$. Let $Q_k = [v_1, \dots, v_k]$ be an orthonormal basis of these eigenvectors and let $D = \mathrm{diag}(\lambda_1, \dots, \lambda_k)$ where $\lambda_1 \geq \lambda_2 \geq \cdots \geq \lambda_k > 0$. Then, sLMP  can be written as
\begin{align*}
H_k &= I_n - Q_k Q_k^T + \theta Q_k D^{-1} Q_k^T = I_n - \sum_{i=1}^k \left(1 - \frac{\theta}{\lambda_i}\right) v_i v_i^T.
\end{align*}

The preconditioned operator $H_k A$ has eigenvalues 
$\underbrace{\theta, \dots, \theta}_{k \text{ times}}, \quad \lambda_{k+1}, \dots, \lambda_n,$
so that the dominant eigenvalues are clustered at $\theta$.

For implementation purposes, the preconditioner is often applied as a sequence of rank-one updates, leading to the equivalent factorized form
\begin{equation} \label{rank-one form}
H_k = \prod_{i=1}^k (I_n - \alpha_i v_i v_i^T), \quad \text{with} \quad \alpha_i = 1 - \frac{\theta}{\lambda_i}.
\end{equation}
This representation is preferred over the corresponding compact implementation because of its improved numerical stability~\cite{gratton2011class}. It also reveals that the action of the preconditioner depends explicitly on the coefficients $\alpha_i$, whose magnitude is governed by the choice of $\theta$ and directly influences the propagation of errors.

Diouane et al.~\cite{diouane2024efficient} studied the choice of $\theta$
from an exact-arithmetic convergence perspective. They showed that, for
$\theta \in [\lambda_{k+1},\lambda_k]$, at each iteration
\[
    \|x^*-\widehat{x}_{\ell}(\theta)\|_A
    \leq
    \|x^*-x_{\ell}\|_A,
\]
for any initial guess, where $x_\ell$ denotes the
$\ell$-th iterate generated by CG without preconditioning and
$\widehat{x}_{\ell}(\theta)$ denotes the $\ell$-th iterate generated by
preconditioned CG using sLMP with cluster point $\theta$, assuming that
$\widehat{x}_{0}(\theta)=x_0$. Moreover, within this interval, $\theta=\lambda_{k+1}$ is the optimal choice from an exact-arithmetic convergence perspective. They also proposed cluster points motivated by improving the early convergence of CG. In particular, they derived the residual-dependent choice
\begin{equation}
\label{eq:theta_r}
    \theta_r
    =
    \frac{\displaystyle\sum_{i=k+1}^{n}\lambda_i\eta_i^2}
         {\displaystyle\sum_{i=k+1}^{n}\eta_i^2},
    \qquad
    \eta_i=v_i^{\top}r_0,
    \qquad
    r_0=b-Ax_0
\end{equation}
which uniquely minimizes the energy norm of the error after the first preconditioned CG iteration. This choice satisfies $\lambda_n \leq \theta_r \leq \lambda_{k+1}$
and may be interpreted as the center of mass for the unmodified part of the spectrum, with weights determined by the components of the initial residual $r_0$.

A further choice is obtained by relating the scaled spectral preconditioner to deflated CG~\cite{diouane2024efficient}. The resulting upper bound relates the error of preconditioned CG to that of deflated CG, whose multiplicative factor is minimized at
 \begin{equation}
 \label{eq:theta_m}
    \theta_m
    =
    \frac{\lambda_{k+1}+\lambda_n}{2}.
\end{equation}
Thus, the choices $\lambda_k$, $\lambda_{k+1}$, $\theta_r$, and $\theta_m$ provide useful reference cluster points derived from different exact-arithmetic convergence considerations.

The above results are derived in exact arithmetic. In the following section, we investigate how finite-precision effects depend on the choice of the cluster point, considering both floating-point rounding errors in the application of sLMP and inaccuracies in the spectral information used to construct it.

\section{Rounding-Error and Sensitivity Analysis}
\label{sec: Analysis}
Let the vector $y \in \mathbb{R}^n$ be expressed as
\begin{equation}
\label{y_D+y_C}
    y = \sum_{i=1}^n \beta_i v_i
=
\underbrace{\sum_{i=1}^k \beta_i v_i}_{y_D}
+
\underbrace{\sum_{i=k+1}^n \beta_i v_i}_{y_C}
\end{equation}

(i.e., expanded in the eigenbasis of $A$), where $\{v_i\}_{i=1}^n$ are the orthonormal eigenvectors associated with the ordered eigenvalues $\{\lambda_i\}_{i=1}^n$. Here, we denote by $y_D$  the components in the \textit{dominant}  subspace, while $y_C$ denotes the components in the \textit{complement} subspace. Throughout this
section, $\|\cdot\|$ denotes the Euclidean norm $\|\cdot\|_2$.

The action of sLMP on $y$ in exact arithmetic is given by
\begin{align}\label{sLMP exact arithmetic}
    H_ky &=\sum_{i=1}^k \frac{\theta}{\lambda_i}\beta_i v_i
+
\sum_{i=k+1}^n \beta_i v_i.\\
&=H_ky_D+y_C \label{eq:exact sLMP action on y}
\end{align}

Thus, sLMP acts as a scaling operator on the \textit{dominant} subspace while leaving the \textit{complement} subspace unchanged. 
Therefore, in exact arithmetic, the cluster point \(\theta\) directly controls the action of sLMP on the \textit{dominant} subspace, while the complement subspace is left unchanged.

We consider two main sources of error affecting the practical behaviour of sLMP. The first arises from floating-point rounding errors
generated during the application of the rank-one updates defining the
preconditioner in~(\ref{rank-one form}). The second arises from inaccuracies in
the spectral information used to construct the preconditioner. We analyze
these effects separately in Subsections~\ref{Rounding error Analysis} and~\ref{Perturbation analysis}, respectively.

\begin{definition}
\label{Definition}
Following the standard floating-point error analysis of Higham 
\cite[Chapter~3]{doi:10.1137/1.9780898718027}, let
\[
    \gamma_n = \frac{nu}{1-nu},
\]
where $u$ denotes the unit roundoff and $nu<1$. Then, for any 
$a,b\in\mathbb{R}^n$, we use the standard floating-point model
\[
    \operatorname{fl}(a^T b)
    = a^T b + \delta,
    \qquad
    |\delta|
    \leq
    \gamma_n \|a\| \|b\|.
\]
\end{definition}
\subsection{Rounding-Error Analysis}
\label{Rounding error Analysis}
\begin{figure}[!htbp]
    \centering
    \includegraphics[width=0.6\linewidth]{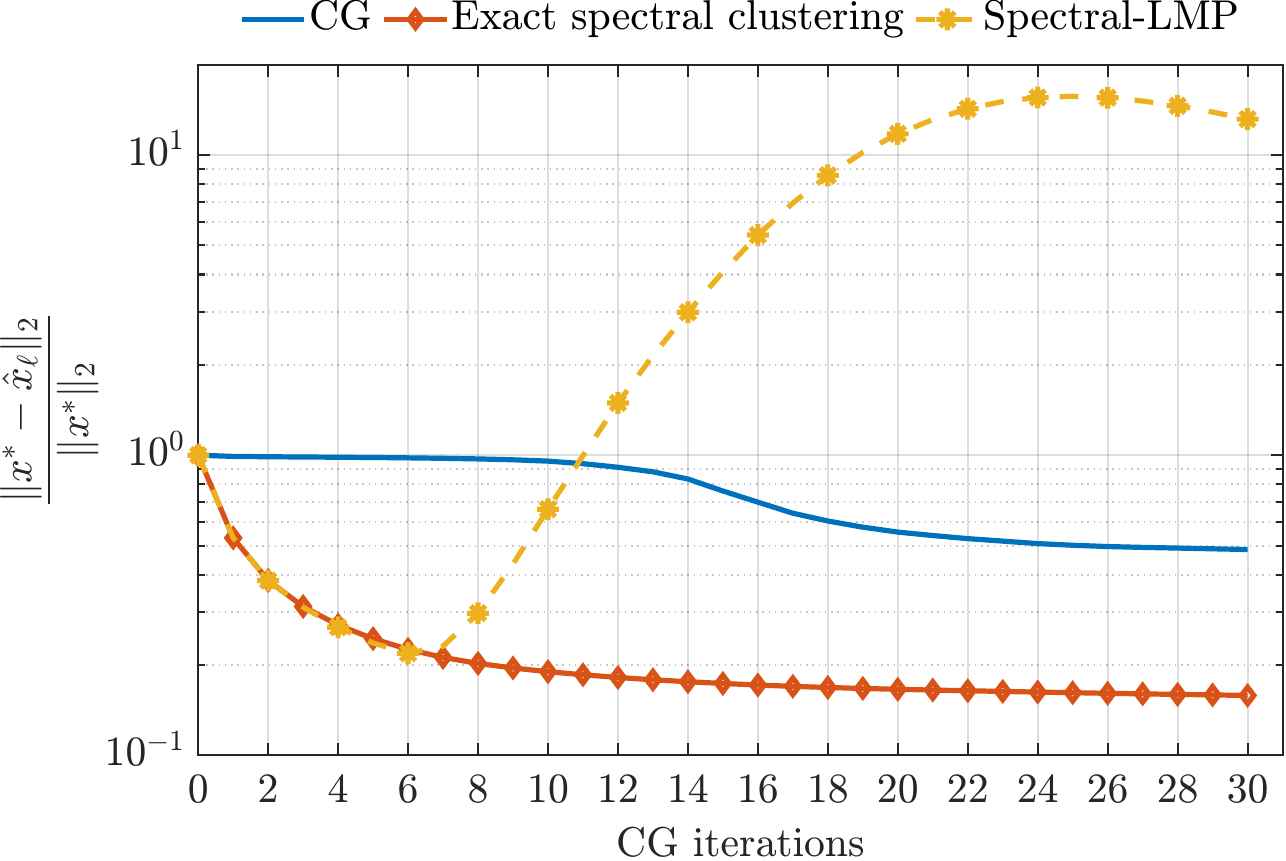}
    \caption{Relative forward-error convergence for unpreconditioned CG,
exact spectral clustering, and practical sLMP with $\theta=1$.}
    \label{fig:LMPvsExactLMP}
\end{figure}
Before analysing the propagation of rounding errors in  sLMP, we illustrate
their relevance through a numerical example. We consider an SPD matrix
$A=Q\Lambda Q^T$ of size $n=5000$, where $Q$ is obtained from the QR
factorization of a random Gaussian matrix and
$\Lambda=\operatorname{diag}(\lambda_1,\ldots,\lambda_n)$. We choose the first $k=100$
eigenvalues to decrease linearly from $10^8$ to $10^6$, and the remaining $n-k$ eigenvalues from $10^5$ to $1$. The exact solution $x^\ast$ is sampled from
$\mathcal{N}(0,I)$ and $b=Ax^\ast$. The first $k$ eigenpairs are taken directly
from this decomposition, in order to eliminate inaccuracies in the dominant spectral
information.

We compare unpreconditioned CG, preconditioned CG using the practical sLMP
with $\theta=1$, and CG applied to the exact spectral modification
$A_{\mathrm{ex}}=Q\Lambda_{\mathrm{ex}}Q^T$, where
\[
\Lambda_{\mathrm{ex}}
=
\operatorname{diag}
(\underbrace{1,\ldots,1}_{k\text{ times}},
\lambda_{k+1},\ldots,\lambda_n),
\qquad
b_{\mathrm{ex}}=A_{\mathrm{ex}}x^\ast.
\]
Thus, $A_{\mathrm{ex}}$ produces exactly the eigenvalue clustering intended by
sLMP with $\theta=1$, without applying the rank-one updates
in~\eqref{rank-one form}, while preserving the exact solution $x^\ast$. Figure~\ref{fig:LMPvsExactLMP} shows markedly different convergence behaviour.
Exact spectral clustering improves convergence relative to unpreconditioned
CG. Although the practical sLMP and exact spectral modification have the same spectrum in exact arithmetic, the practical sLMP initially follows the exact modification, but after approximately eight iterations its relative forward error begins to increase substantially.
Since spectral inaccuracies have been excluded, this discrepancy motivates
the following analysis of rounding errors and their dependence on~$\theta$.

To analyse these rounding errors, we consider separately the action of the
sLMP on the dominant and complement components, $y_D$ and $y_C$, defined
in~\eqref{y_D+y_C}. 
\begin{lemma}\label{complement lemma}
Let ${v_1,\dots,v_k,\cdots,v_n}\in\mathbb{R}^n$ be an orthonormal set of vectors. Define $\alpha_i = 1 - \frac{\theta}{\lambda_i}$, and 
$H_ky_C =
\prod_{i=1}^k (I_n - \alpha_i v_i v_i^{T})y_C,
$ 
where $y_C \in \mathrm{span}\{v_{k+1},\dots,v_n\}$. Then, to first order in the unit roundoff $u$, the
rounding error satisfies
\begin{equation}
\label{eq:complement rounding error}
\operatorname{fl}(H_ky_C)=y_C \circ \left(\mathbf{1}+\sum_{i=1}^k \mu_i\right)-\sum_{i=1}^k \alpha_i \varepsilon_i v_i-\sum_{i=1}^k \alpha_i v_i^T \left(y_c \circ \sum_{j=1}^{i-1} \mu_j\right)v_i+\mathcal{O}(u^2),
\end{equation}
where $\circ$ denotes the Hadamard (elementwise) product, i.e.,
$(a\circ b)_\ell = a_\ell b_\ell$ for $\ell=1,\dots,n$, $\mu_i\in\mathbb{R}^n$ is the componentwise relative error associated with the $i$-th rank-one update, and
$\mathbf{1}=(1,\ldots,1)^{T}$.

The quantities $\varepsilon_i$ and $\mu_i$ satisfy
\begin{equation}\label{eq:FlopRelations}
\varepsilon_i=\operatorname{fl}(v_i^Ty_C),
\qquad
|\varepsilon_i| \le \gamma_n \|y_C\|,
 \qquad \|\mu_i\|_\infty \le u.    
\end{equation}
\end{lemma}
\begin{proof}
\noindent We proceed by induction,\\ \textbf{Base case:} When \(k = 1\),
\begin{align*}
    \operatorname{fl}(H_1y_C)&=\operatorname{fl}((I-\alpha_1v_1v_1^T)y_C)\\
    &=(y_C-\alpha_1 \varepsilon_1 v_1) \circ (\mathbf{1}+\mu_1)+\mathcal{O}(u^2)
\end{align*}
Substituting $k=1$ in~\eqref{eq:complement rounding error} shows that the base case holds.

\noindent\textbf{Inductive step:} Assume~\eqref{eq:complement rounding error} holds for \(k = m\), then we show it also holds for \(k = m+1\): 
\begin{align*}
\operatorname{fl}(H_{m+1}y_C)
&= \operatorname{fl}
\left[
\operatorname{fl}(H_m y_C)
-
\alpha_{m+1}v_{m+1}
\operatorname{fl}
\left(
v_{m+1}^{T}\operatorname{fl}(H_m y_C)
\right)
\right].
\end{align*}

Using the induction hypothesis and \eqref{eq:FlopRelations} 
we obtain
\[
\operatorname{fl}\!\left(
v_{m+1}^T\operatorname{fl}(H_my_C)
\right)
=
\varepsilon_{m+1}
+
v_{m+1}^T
\left(
y_C\circ\sum_{j=1}^{m}\mu_j
\right)
+O(u^2).
\]
Hence,
\[
\begin{aligned}
\operatorname{fl}(H_{m+1}y_C)
={}&
\Bigg[
\operatorname{fl}(H_my_C)
-\alpha_{m+1}\varepsilon_{m+1}v_{m+1} \\
&\qquad
-\alpha_{m+1}
v_{m+1}^T
\left(
y_C\circ\sum_{j=1}^{m}\mu_j
\right)v_{m+1}
\Bigg]
\circ(1+\mu_{m+1})
+O(u^2).
\end{aligned}
\]

Substituting the induction hypothesis and neglecting products of first-order \\terms yields~\eqref{eq:complement rounding error} for $k=m+1$.
\end{proof}

Lemma~\ref{complement lemma} shows that the relative forward error depends on the particular 
realizations of the rounding errors $\varepsilon_i$ and $\mu_i$. 
Minimizing the relative forward error directly yields a 
cluster point that depends on these realizations, which are not known 
a priori. We now present a realization-independent choice of $\theta$ 
by deriving an upper bound on the relative forward error that holds for all 
admissible rounding errors.
\begin{theorem}
\label{complement theorem}
Under the assumptions of Lemma~\ref{complement lemma}, the relative forward error in
the application of  sLMP to
$y_C\in\operatorname{span}\{v_{k+1},\ldots,v_n\}$ satisfies, to first
order in the unit roundoff $u$,
$\frac{\|y_C-\operatorname{\operatorname{fl}}(H_ky_C)\|}{\|y_C\|}
\le
(\gamma_n+ku)
\sqrt{
\sum_{i=1}^k
\alpha_i^2 }
+ku+O(u^2).
$ 

The first-order upper bound is uniquely minimized at
\begin{equation}
\label{weighted mean}
\theta_{\mathrm{WAM}}
=
\frac{\displaystyle\sum_{i=1}^k 1/\lambda_i}
     {\displaystyle\sum_{i=1}^k 1/\lambda_i^{2}},
\end{equation}
which is the weighted arithmetic mean of the dominant eigenvalues
$\lambda_1,\ldots,\lambda_k$ with weights $w_i=\frac{1}{\lambda_i^2}$.
\end{theorem}
\begin{proof}
Using~\eqref{eq:exact sLMP action on y} and~\eqref{eq:complement rounding error} we obtain the bound
\begin{align*}
    \|y_C-\operatorname{fl}(H_ky_C)\| 
    & \leq \|y_C\| \left\|\sum_{i=1}^k \mu_i\right\|_\infty + \sqrt{\sum_{i=1}^k \alpha_i^2 \varepsilon_i^2} + \sqrt{\sum_{i=2}^k \alpha_i^2 \|y_C\|^2 \left\|\sum_{j=1}^{i-1} \mu_j\right\|_\infty^2}. 
\end{align*}
Applying \eqref{eq:FlopRelations} yields
\begin{align*}
    \frac{\|y_C-\operatorname{fl}(H_ky_C)\|}{\|y_C\|}
    & \leq (\gamma_n+ku)\sqrt{\sum_{i=1}^k \alpha_i^2} + ku +\mathcal{O}(u^2)
\end{align*}
Finally, recalling the definition of $\alpha_i$~\eqref{rank-one form}, the minimiser of the bound is simply the minimiser of the least-squares problem,
\begin{equation*}
    f(\theta)=\sum_{i=1}^k \left(1-\frac{\theta}{\lambda_i}\right)^2
\end{equation*}
since $k$, $u$, and $\gamma_n$ are independent of $\theta$. This least-squares problem has the minimum given by the theorem statement.
\end{proof}

Now we analyse the propagation of rounding errors when applying sLMP on $y_D$ by giving an upper bound for the computed action of $H_k$ on $y_D$ in finite precision.
\begin{theorem}
\label{Dominant theorem}
    Let $\{v_1,\dots,v_k\}\subset\mathbb{R}^n$, $k\leq n$, be an orthonormal set of vectors. 
Let $\beta_i\in\mathbb{R}$ and define $\alpha_i = 1 - \frac{\theta}{\lambda_i}$ where $\theta$ is the cluster point. Then
\[
 \frac{\|H_ky_D-\operatorname{fl}(H_ky_D)\|}{\|H_ky_D\|} \leq C_D(\theta) \left(ku + \gamma_n \sum_{i=1}^k |\alpha_i|\right) + \mathcal{O}(u^2),
\]
where
\[
 C_D(\theta)= \begin{cases}
      \frac{\theta}{\lambda_k}, & \theta \geq \lambda_1\\
      \frac{\lambda_1}{\lambda_k}, & \lambda_k \leq \theta < \lambda_1 \\
      \frac{\lambda_1}{\theta}, & 0<\theta < \lambda_k
    \end{cases}.
\]
\end{theorem}
\begin{proof}
Let $F_i=I_n-\alpha_i v_i v_i^T$, and define
\begin{equation}
\label{eq: ith rank-one update in fl}
\widehat{y}_D^{(i)}
=
\operatorname{fl}\left(F_i\widehat{y}_D^{(i-1)}\right)
=
F_i\widehat{y}_D^{(i-1)}+r_i,
\qquad
\widehat{y}_D^{(0)}=y_D,
\end{equation}
where $r_i$ denotes the local rounding error introduced at the $i$th update.

We first bound $r_i$. By Definition~\ref{Definition},
\[
\operatorname{fl}\left(v_i^T\widehat{y}_D^{(i-1)}\right)
=
v_i^T\widehat{y}_D^{(i-1)}+\varepsilon_i,
\qquad
|\varepsilon_i|
\leq
\gamma_n\|\widehat{y}_D^{(i-1)}\|.
\]
Hence,
\[
\operatorname{fl}(F_i\widehat y_D^{(i-1)})
=
\left[
\widehat y_D^{(i-1)}
-\alpha_i
\left(v_i^T\widehat y_D^{(i-1)}+\varepsilon_i\right)v_i
\right]\circ(1+\mu_i)+O(u^2),
\]
where $\|\mu_i\|_\infty\leq u$. Since
$F_i\widehat y_D^{(i-1)}
=
\widehat y_D^{(i-1)}
-\alpha_i v_i^T\widehat y_D^{(i-1)}v_i$,
\[
r_i
=
F_i\widehat y_D^{(i-1)}\circ\mu_i
-\alpha_i\varepsilon_i v_i
+O(u^2).
\]
Using $\|a\circ b\|\leq\|a\|\|b\|_\infty$ gives
\begin{align}
\|r_i\|
&\leq
u\|F_i\widehat y_D^{(i-1)}\|
+
\gamma_n|\alpha_i|
\|\widehat y_D^{(i-1)}\|
+O(u^2).
\label{eq:ri-local}
\end{align}
Since
$\widehat y_D^{(i-1)}=y_D^{(i-1)}+O(u)$ and
$F_i\widehat y_D^{(i-1)}=y_D^{(i)}+O(u)$, and since $u$ and
$\gamma_n$ are already first-order quantities,
\begin{equation}
\|r_i\|
\leq
u\|y_D^{(i)}\|
+
\gamma_n|\alpha_i|
\|y_D^{(i-1)}\|
+O(u^2).
\label{eq:ri-bound}
\end{equation}

We now bound the accumulated rounding error
$e_k=\widehat y_D^{(k)}-H_ky_D$. Recursively expanding
\eqref{eq: ith rank-one update in fl} and using
$H_ky_D=\prod_{i=1}^kF_i y_D$ gives
\[
e_k
=
\sum_{i=1}^{k}
\left(\prod_{j=i+1}^{k}F_j\right)r_i
+O(u^2).
\]
Since $\prod_{j=i+1}^{k}F_j$ is SPD with eigenvalues
$\{\theta/\lambda_j,1\}$,
\[
\left\|\prod_{j=i+1}^{k}F_j\right\|
\leq
\max\left\{1,\frac{\theta}{\lambda_k}\right\}
=:M(\theta).
\]
Therefore, using \eqref{eq:ri-bound},
\begin{equation}
\frac{\|e_k\|}{\|H_k y_D\|}
\leq
M(\theta)
\frac{\sum_{i=1}^{k}
\left(
u\|y_D^{(i)}\|
+
\gamma_n|\alpha_i|\|y_D^{(i-1)}\|
\right)}
{\|H_k y_D\|}
+O(u^2).
\label{eq:error-before-C}
\end{equation}

It remains to bound the intermediate vectors. They satisfy
\[
y_D^{(i)}
=
\sum_{j=1}^{i}\frac{\theta}{\lambda_j}\beta_jv_j
+
\sum_{j=i+1}^{k}\beta_jv_j,
\qquad
H_k y_D
=
\sum_{j=1}^{k}\frac{\theta}{\lambda_j}\beta_jv_j.
\]
Thus,
\[
\|y_D^{(i)}\|^2
=
\sum_{j=1}^{i}\frac{\theta^2\beta_j^2}{\lambda_j^2}
+
\sum_{j=i+1}^{k}
\frac{\theta^2\beta_j^2}{\lambda_j^2}
\frac{\lambda_j^2}{\theta^2},
\qquad
\|H_k y_D\|^2
=
\sum_{j=1}^{k}\frac{\theta^2\beta_j^2}{\lambda_j^2}.
\]
If $\theta\geq\lambda_1$, then $\lambda_j/\theta\leq1$, whereas if
$0<\theta<\lambda_1$, then
$\lambda_j/\theta\leq\lambda_1/\theta$. Hence,
\[
\frac{\|y_D^{(i)}\|}{\|H_k y_D\|}
\leq
\max\left\{1,\frac{\lambda_1}{\theta}\right\}.
\]
The same argument applies to $y_D^{(i-1)}$, so
\[
\frac{\|y_D^{(i)}\|}{\|H_k y_D\|},
\quad
\frac{\|y_D^{(i-1)}\|}{\|H_k y_D\|}
\leq
C(\theta),
\qquad
C(\theta)
=
\max\left\{1,\frac{\lambda_1}{\theta}\right\}.
\]
Applying this bound to \eqref{eq:error-before-C} gives
\[
\frac{\|H_k y_D-\operatorname{fl}(H_k y_D)\|}
     {\|H_k y_D\|}
\leq
M(\theta)C(\theta)
\left(
ku+\gamma_n\sum_{i=1}^{k}|\alpha_i|
\right)
+O(u^2).
\]

Finally,
\[
M(\theta)C(\theta)
=
\max\left\{1,\frac{\theta}{\lambda_k}\right\}
\max\left\{1,\frac{\lambda_1}{\theta}\right\}
=
\begin{cases}
\dfrac{\theta}{\lambda_k},
& \theta\geq\lambda_1,\\[2mm]
\dfrac{\lambda_1}{\lambda_k},
& \lambda_k\leq\theta<\lambda_1,\\[2mm]
\dfrac{\lambda_1}{\theta},
& 0<\theta<\lambda_k.
\end{cases}
\]
Substitution yields the stated bound.
\end{proof}

\begin{corollary}
\label{cor:Dominant corollary}
The first-order bound in Theorem~\ref{Dominant theorem} is minimized by
taking $\theta$ to be a weighted median of
$\{\lambda_1,\ldots,\lambda_k\}$ with weights $w_i=1/\lambda_i$.
\end{corollary}

\begin{proof}
For $0<\theta<\lambda_k$, the first-order bound in
Theorem~\ref{Dominant theorem} is strictly decreasing, while for
$\theta>\lambda_1$ it is strictly increasing. Hence, a global minimizer lies in
$[\lambda_k,\lambda_1]$. On this interval,
$C_D(\theta)=\lambda_1/\lambda_k$ is independent of $\theta$, so
minimizing the first-order bound is equivalent to minimizing
\[
\sum_{i=1}^k
\left|1-\frac{\theta}{\lambda_i}\right|
=
\sum_{i=1}^k
\frac{1}{\lambda_i}|\lambda_i-\theta|.
\]
This is a weighted absolute-deviation problem, whose minimizers are the
weighted medians of $\{\lambda_1,\ldots,\lambda_k\}$ with weights
$w_i=1/\lambda_i$.
\end{proof}

Theorem~\ref{Dominant theorem} shows that rounding-error amplification in the
dominant subspace depends strongly on the position of $\theta$ relative to the
dominant spectrum. For $\theta<\lambda_k$ and $\theta>\lambda_1$, the
amplification is governed by $\lambda_1/\theta$ and $\theta/\lambda_k$,
respectively, whereas for $\theta\in[\lambda_k,\lambda_1]$ it remains constant
at $\lambda_1/\lambda_k$. Within this interval, the dependence on \(\theta\) is therefore governed by the coefficients \(\alpha_i\), and Corollary~\ref{cor:Dominant corollary} shows that the resulting bound is minimized by the weighted-median choice $\theta_{WM,1/\lambda}$.

\subsection{Sensitivity Analysis}
\label{Perturbation analysis}
The analysis above assumes nearly exact spectral information. In practice,
the dominant eigenpairs used to construct sLMP are computed numerically and
may therefore be inaccurate. We now analyse the sensitivity of sLMP to perturbations in the dominant
eigenpairs.

\begin{proposition}\label{prop: perturbed preconditioner}
Let
\begin{equation} \label{eq:Hk}
H_k
=
I-\sum_{i=1}^k \alpha_i v_i v_i^T,
\qquad
\alpha_i=1-\frac{\theta}{\lambda_i},    
\end{equation}

 Suppose the computed eigenvectors and eigenvalues of $A$ are perturbed such that
\begin{equation}\label{eq:Perturbations}
    \widehat v_i=v_i+\Delta v_i,
\qquad
\widehat\lambda_i=\lambda_i+\Delta\lambda_i,
\end{equation}

with
\begin{equation}\label{eq:PerturbationNorms}
   \|v_i\|=\|\widehat v_i\|=1,
\qquad
\|\Delta v_i\|= \tau_i,
\qquad
\left|\frac{\Delta\lambda_i}{\lambda_i}\right|
= \zeta_i,
\quad i=1,\ldots,k. 
\end{equation}

We define the perturbed quantities
\begin{equation}\label{eq:Perturbed alpha H}
  \widehat\alpha_i
=
1-\frac{\theta}{\widehat\lambda_i},
\qquad
\widehat H_k
=
I-\sum_{i=1}^k \widehat\alpha_i \widehat v_i \widehat v_i^T.
\end{equation}

Then
$
\|\widehat H_k-H_k\|
\le
\sum_{i=1}^k
\left(
\frac{\theta\zeta_i}{\widehat\lambda_i}
+
|\widehat\alpha_i|
(2\tau_i+\tau_i^2)
\right).
$
\end{proposition}
\begin{proof}
From \eqref{eq:Perturbations} and \eqref{eq:PerturbationNorms}
we have
\[
|\widehat\alpha_i-\alpha_i|
=
\left|
1-\frac{\theta}{\widehat\lambda_i}
-
1+\frac{\theta}{\lambda_i}
\right|
=
\theta
\left|
\frac{\widehat\lambda_i-\lambda_i}
{\lambda_i\widehat\lambda_i}
\right|
=
\frac{\theta\zeta_i}{\widehat\lambda_i}.
\]

Using the definitions of $H_k$ and $\widehat H_k$, \eqref{eq:Hk} and \eqref{eq:Perturbed alpha H}
\begin{align*}
\|\widehat H_k-H_k\|
&=
\left\|
\sum_{i=1}^k
\left[
\alpha_i v_i v_i^T
-
\widehat\alpha_i
\left(
v_i v_i^T
+
v_i\Delta v_i^T
+
\Delta v_i v_i^T
+
\Delta v_i\Delta v_i^T
\right)
\right]
\right\| \\
&\le
\sum_{i=1}^k
\Big(
|\alpha_i-\widehat\alpha_i|
\|v_i\|\|v_i^T\|
\nonumber\\
&\qquad\qquad
+
|\widehat\alpha_i|
\big(
\|v_i\|\|\Delta v_i^T\|
+
\|\Delta v_i\|\|v_i^T\|
+
\|\Delta v_i\|\|\Delta v_i^T\|
\big)
\Big).
\end{align*}
Using \eqref{eq:PerturbationNorms} and \eqref{eq:Perturbed alpha H} we obtain the result in the theorem statement.
\end{proof}
Proposition~\ref{prop: perturbed preconditioner} bounds the difference between the perturbed sLMP operator and the operator constructed using the exact dominant eigenpairs. We next translate this operator bound into a relative application-error bound.

\begin{corollary}\label{cor:perturbed-relative}
Under the assumptions of Proposition~\ref{prop: perturbed preconditioner},
for any $y$ such that $H_ky\neq 0$,
\[
\frac{\|(\widehat H_k-H_k)y\|}{\|H_ky\|}
\le
C(\theta)\,
\sum_{i=1}^k \left( \cfrac{\theta \zeta_i}{\widehat \lambda_i}
+w_i \left|1-\cfrac{\theta}{\widehat \lambda_i}\right|\right),
\]
where $w_i=2\tau_i+\tau_i^2, i=1,\ldots,k,$ and
\begin{equation}
 \label{C(theta)}
C(\theta)
=
\begin{cases}
1, & \theta\ge \lambda_1,\\[1ex]
\dfrac{\lambda_1}{\theta}, & 0<\theta< \lambda_1.
\end{cases}
\end{equation}
\end{corollary}
\begin{proof}
Using the definition of the induced matrix norm,
and dividing by $\|H_ky\|$,
\[
\frac{\|(\widehat H_k-H_k)y\|}{\|H_ky\|}
\le
\|\widehat H_k-H_k\|
\frac{\|y\|}{\|H_ky\|}.
\]

Combining the inequality $\frac{\|y\|}{\|H_ky\|}
\le
C(\theta)$ with Proposition~\ref{prop: perturbed preconditioner}
yields
\[
\frac{\|(\widehat H_k-H_k)y\|}{\|H_ky\|}
\le
C(\theta)\,
\sum_{i=1}^k \left( \cfrac{\theta \zeta_i}{\widehat \lambda_i}+w_i \left|1-\cfrac{\theta}{\widehat \lambda_i}\right|\right).
\]
\end{proof}
The following theorem determines the choice of cluster point $\theta$ 
that minimizes the bound on the relative perturbation in the action of the preconditioner derived in Corollary~\ref{cor:perturbed-relative}.

\begin{theorem}
\label{theorem: weighted median}
Assume that $\lambda_1\ge\widehat\lambda_1\ge\widehat\lambda_2
\ge\cdots\ge\widehat\lambda_k>0$.
Then the perturbation bound in
Corollary~\ref{cor:perturbed-relative} is minimized by taking
$\theta^\star$ to be a weighted median of the perturbed dominant eigenvalues
$\{\widehat\lambda_1,\ldots,\widehat\lambda_k\}$.
Equivalently, one such minimizer is
$\theta^\star=\widehat\lambda_j$, where $j$ is the smallest index satisfying
\[
\sum_{i=1}^{j-1}w_i
\le
\frac{W}{2}
\le
\sum_{i=1}^{j}w_i,
\qquad
W=\sum_{i=1}^{k}w_i.
\]
\end{theorem}
\begin{proof} Let \[ f(\theta) = C(\theta) \sum_{i=1}^{k} \left( \frac{\theta\zeta_i}{\widehat\lambda_i} + w_i \left|1-\frac{\theta}{\widehat\lambda_i}\right| \right). \] We consider two cases for $\theta$. \paragraph{Case 1: For $\theta>\lambda_1$} Since $\theta>\lambda_1\ge\widehat\lambda_1 \ge\widehat\lambda_i, i=1,\ldots,k,$ we have $ \left|1-\frac{\theta}{\widehat\lambda_i}\right| = \frac{\theta}{\widehat\lambda_i}-1.$ Using $C(\theta)=1$, we obtain $ f(\theta) = \sum_{i=1}^{k} \left[ \frac{\theta\zeta_i}{\widehat\lambda_i} + w_i \left( \frac{\theta}{\widehat\lambda_i}-1 \right) \right]. $ Therefore,  $f'(\theta) = \sum_{i=1}^{k} \frac{\zeta_i+w_i}{\widehat\lambda_i} > 0.$ Hence $f(\theta)$ is increasing for $\theta>\lambda_1$, so no minimizer can occur strictly above $\lambda_1$. \paragraph{Case 2: For $0<\theta\le\lambda_1$}
we have \begin{align*} f(\theta) &= \frac{\lambda_1}{\theta} \sum_{i=1}^{k} \left( \frac{\theta\zeta_i}{\widehat\lambda_i} + w_i \left|1-\frac{\theta}{\widehat\lambda_i}\right| \right) \\ &= \lambda_1 \sum_{i=1}^{k} \frac{\zeta_i}{\widehat\lambda_i} + \lambda_1 \sum_{i=1}^{k} w_i \left| \frac{1}{\theta} - \frac{1}{\widehat\lambda_i} \right|. \end{align*} The first term is independent of $\theta$, and the factor $\lambda_1$ does not affect the minimizing value of $\theta$. Therefore, minimizing $f(\theta)$ is equivalent to minimizing $ g(\theta) = \sum_{i=1}^{k} w_i \left| \frac{1}{\theta} - \frac{1}{\widehat\lambda_i} \right|.$ This weighted absolute-deviation function is minimized by a weigh-ted  median of $\left\{ \widehat\lambda_1, \ldots, \widehat\lambda_k \right\}.$ Hence $\theta^\star=\widehat\lambda_j,$ where $j$ is the smallest index satisfying \[ \sum_{i=1}^{j-1}w_i \le \frac{W}{2} \le \sum_{i=1}^{j}w_i. \] Since $\widehat\lambda_j\le\lambda_1$ and $f(\theta)$ is increasing
for $\theta>\lambda_1$, $\theta^\star=\widehat\lambda_j$ is a global
minimizer of the perturbation bound.
\end{proof}
Theorem~\ref{theorem: weighted median} considers perturbations in both
the dominant eigenvectors and eigenvalues. We now consider the special
case in which the dominant eigenvalues are exact and only the eigenvectors
are perturbed. In this case, the weighted-median result reduces to the
following corollary.
\begin{corollary}
\label{cor:exact-eigenvalues}
Under the assumptions of Theorem~\ref{theorem: weighted median}, suppose
that the dominant eigenvalues are exact, so that
\[
\widehat\lambda_i=\lambda_i,
\qquad
\zeta_i=0,
\quad i=1,\ldots,k.
\]
Then the cluster point that minimizes the perturbation bound is the
weighted median of the dominant eigenvalues
$\{\lambda_1,\ldots,\lambda_k\}$. Equivalently,
$\theta^\star=\lambda_j,$
where $j$ is the smallest index satisfying
$
\sum_{i=1}^{j-1}w_i
\le
\frac{W}{2}
\le
\sum_{i=1}^{j}w_i,
\qquad
W=\sum_{i=1}^{k}w_i.
$
\end{corollary}

\begin{proof}
Setting $\widehat\lambda_i=\lambda_i$ and $\zeta_i=0$ in
Theorem~\ref{theorem: weighted median} gives the result directly.
\end{proof}

Theorem~\ref{theorem: weighted median} shows that, in the presence of perturbations in both the dominant eigenvectors and eigenvalues, the sensitivity-oriented cluster point, denoted by \(\theta_{\mathrm{WM},e}\), is a weighted median of the perturbed dominant eigenvalues, with weights determined by the eigenvector perturbation magnitudes. Corollary~\ref{cor:exact-eigenvalues} gives the corresponding
result for exact dominant eigenvalues. Thus, the weighted-median choice
incorporates information about the accuracy of the spectral information used
to construct the preconditioner.

Overall, the error analyses identify three error-informed cluster points:
the weigh-ted arithmetic mean $\theta_{\mathrm{WAM}}$ for complement-subspace
rounding errors, the weighted median $\theta_{WM,1/\lambda}$ for
dominant-subspace rounding errors, and the perturbation-weigh-ted median for
inaccurate spectral information. These choices complement those motivated by
exact-arithmetic convergence. The following section examines these predictions
numerically, considering both rounding errors and inaccurate spectral
information.
\section{Numerical Results}
\label{sec: results}
This section investigates the influence of finite-precision effects on the
practical performance of sLMP and the choice of cluster point. We first examine
rounding errors when nearly exact spectral information is used, and then
consider inaccuracies in the dominant spectral information. The latter are
studied using both controlled synthetic perturbations and approximate
eigenpairs from practical eigensolvers, with particular attention to
perturbations within the dominant subspace and leakage into its complement.
The experiments assess when cluster points motivated by exact-arithmetic
convergence theory~\cite{diouane2024efficient} remain effective in practice.

\subsection{Experimental setup}
All experiments are performed in MATLAB R2024b on a machine with a
2.40 GHz 13th Gen Intel Core i7-13700H processor (14 cores),
16 GB of RAM, and Windows 11 Enterprise.
The MATLAB code used to generate the numerical results in this paper
is available at
\url{https://github.com/hisham-elzayadi/Spectral-LMP}. We use synthetic symmetric positive definite matrices
of size $n=5000$, with the first $k=100$ eigenpairs regarded as dominant. We
construct
\[
A=Q\Lambda Q^T,
\]
where $Q=[v_1,\ldots,v_n]$ is obtained from the QR factorization of a random
Gaussian matrix and
$\Lambda=\operatorname{diag}(\lambda_1,\ldots,\lambda_n)$.
Unless otherwise stated, the dominant eigenvalues are linearly distributed over
$[10^{16},10^{14}]$ and the remaining eigenvalues decrease linearly from
$10^{12}$ to $1$. The exact solution is sampled as $x\sim\mathcal{N}(0,I)$,
with $b=Ax$. All convergence plots show the relative forward error over the
first 50 CG iterations.

For the synthetic experiments, the approximate dominant basis
$\widehat V_k=[\widehat v_1,\ldots,\widehat v_k]$ is constructed with prescribed
perturbation magnitudes
\[
\tau_i=\|\Delta v_i\|=\|\widehat v_i-v_i\|,
\qquad i=1,\ldots,k,
\]
and separately controlled perturbation directions, followed by
orthonormalization. For the practical experiments, the dominant eigenpairs are
approximated using the randomized Nystr\"om
algorithm~\cite{halko2011finding,dauvzickaite2021randomised} and
REVD-ritzit~\cite{dauvzickaite2021randomised}.

We characterize the perturbations using the heat maps
\[
|\widehat V_k^TQ_k|,
\qquad
|\widehat V_k^TQ_{k+1:n}|,
\]
which measure dominant-subspace mixing and complement-subspace leakage,
respectively, together with
\[
\delta_i=\|Q_k^T\Delta v_i\|,
\qquad
\eta_i=\|Q_{k+1:n}^T\Delta v_i\|.
\]
We also report $\|\widehat V_k^TQ_{k+1:n}\|$, the sine of the largest
principal angle between the exact and approximate dominant
subspaces~\cite{principle-angle}.

We compare unpreconditioned CG with the exact-arithmetic choices
$\lambda_k$, $\lambda_{k+1}$, $\theta_r$, and $\theta_m$, and the
error-informed choices $\theta_{\mathrm{WM},1/\lambda}$, and $\theta_{\mathrm{WM},e}$. While $\theta_{\mathrm{WAM}}$ is only used in complement subspace rounding-error experiment. 
Here $\theta_r$ and $\theta_m$ are defined in \eqref{eq:theta_r} and
\eqref{eq:theta_m}, respectively. For synthetic perturbations, these choices
are evaluated using the exact eigenvalues. For practical eigensolvers,
$\widehat\lambda_k$ replaces $\lambda_k$, while $\lambda_{k+1}$,
$\theta_r$, and $\theta_m$ are retained as exact reference values; the
error-informed choices use the corresponding approximate dominant
eigenvalues.
\subsection{sLMP Rounding Errors}

We numerically assess the rounding-error behaviour predicted by
Theorems~\ref{complement theorem} and~\ref{Dominant theorem}.
The dominant eigenvalues are logarithmically distributed over
$[10^{13},10^{16}]$. To examine cluster points below, within, and above
the dominant spectrum, we use $k$ logarithmically spaced values in
$[1,10^{12}]$, the $k$ dominant eigenvalues, and $k$ logarithmically
spaced values in $[10^{17},10^{19}]$, together with
$\theta_{\mathrm{WAM}}$ and $\theta_{\mathrm{WM},1/\lambda}$.
The dominant eigenvectors are available to machine precision, so that
the observed errors arise from rounding during the application of sLMP.

For the complement subspace, we generate a random
$y_C\in\operatorname{span}\{v_{k+1},\ldots,v_n\}$ with normally distributed
coordinates. Figure~\ref{fig:rounding_comp} compares the bound from
Theorem~\ref{complement theorem} with the observed relative error
\[
\frac{\|y_C-\operatorname{fl}(H_k y_C)\|}{\|y_C\|}.
\]
The bound reproduces the overall dependence on $\theta$ and is minimized
at $\theta_{\mathrm{WAM}}$, as predicted. The observed error remains close
to machine precision for small cluster points, including $\theta=1$,
and increases as $\theta$ moves above the dominant spectrum, although
its magnitude remains comparatively small.

\begin{figure}[htbp]
    \centering
    \begin{subfigure}[t]{0.49\textwidth}
        \centering
        \includegraphics[width=\textwidth]{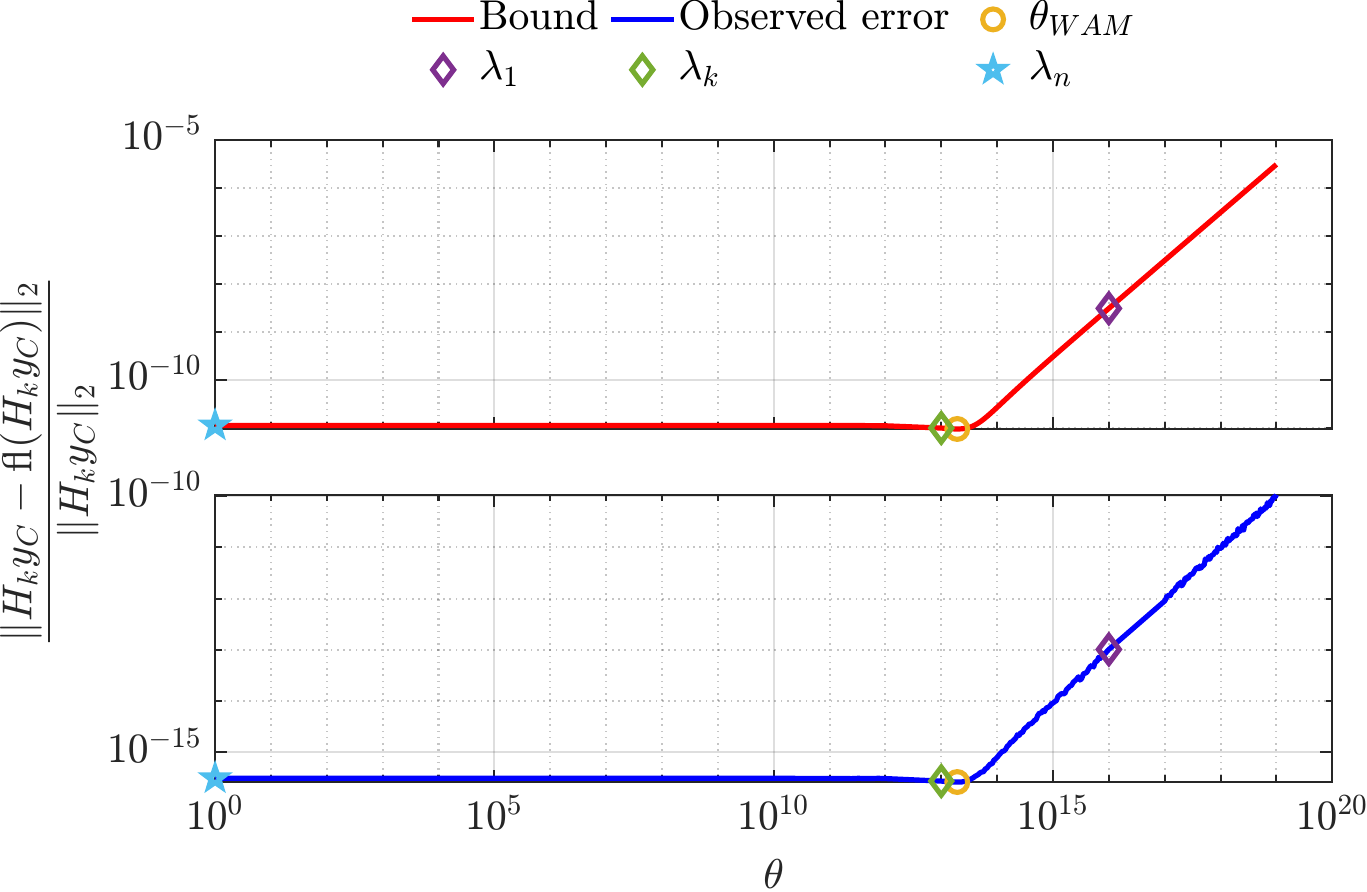}
        \caption{Complement subspace.}
        \label{fig:rounding_comp}
    \end{subfigure}
    \hfill
    \begin{subfigure}[t]{0.49\textwidth}
        \centering
        \includegraphics[width=\textwidth]{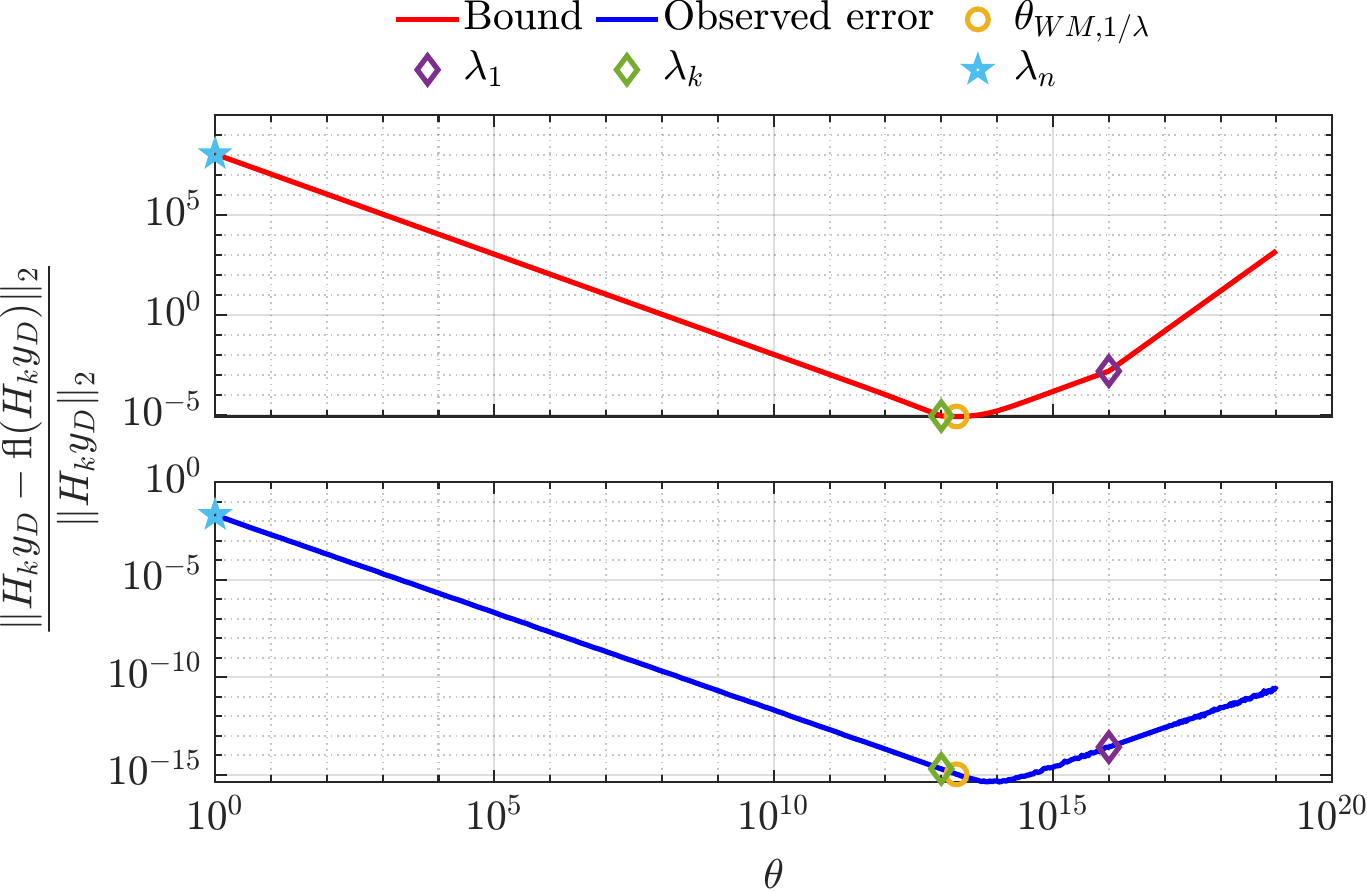}
        \caption{Dominant subspace.}
        \label{fig:rounding_dom}
    \end{subfigure}
    \caption{Theoretical rounding-error bounds and observed relative
    forward errors as functions of $\theta$. The complement-subspace
    bound is minimized at $\theta_{\mathrm{WAM}}$, while the
    dominant-subspace bound is minimized at
    $\theta_{\mathrm{WM},1/\lambda}$.}
    \label{fig:rounding_errors}
\end{figure}

For the dominant subspace, we similarly generate a random
$y_D\in\operatorname{span}\{v_1,..,v_k\}$.
Figure~\ref{fig:rounding_dom} compares the bound from
Theorem~\ref{Dominant theorem} with
\[
\frac{\|H_k y_D-\operatorname{fl}(H_k y_D)\|}
     {\|H_k y_D\|}.
\]
The bound is minimized at $\theta_{\mathrm{WM},1/\lambda}$, which lies
within the observed low-error region. The error remains close to machine
precision throughout much of $[\lambda_k,\lambda_1]$, but increases
rapidly for $\theta<\lambda_k$. In particular, near $\theta=\lambda_n=1$
the dominant-subspace error is several orders of magnitude larger,
explaining the deterioration of the practical sLMP with $\theta=1$
observed in Figure~\ref{fig:LMPvsExactLMP}.

These experiments reveal a marked asymmetry between the two subspaces. Small cluster points are benign in the complement subspace but can strongly amplify rounding errors in the dominant subspace. In contrast, cluster points near \(\lambda_k\) give small errors in both. Thus, the dominant subspace imposes the more restrictive stability requirement, while choices of \(\theta\) near \(\lambda_k\) are compatible with both rounding-error stability and the exact-arithmetic convergence interval $[\lambda_{k+1},\lambda_k]$.
\subsection{Synthetic perturbations}
\label{Result sec: Synthetic perturbation}
We investigate the sensitivity of sLMP to inexact spectral information using
synthetic perturbations, in which case the magnitude and direction of the
eigenvector errors can be controlled independently. Three cases are considered:
large unrestricted perturbations, large perturbations confined to the dominant
subspace, and small unrestricted perturbations. These experiments isolate the
effects of perturbation direction and magnitude.
\subsubsection{Large perturbations with unrestricted directions}
Synthetic perturbations are generated with linearly distributed magnitudes
$\tau_i\in[10^{-6},1]$ and unrestricted directions, allowing both mixing within
the dominant subspace and leakage into its complement.

Figure~\ref{fig:Exp1Conv} shows a clear difference from exact-arithmetic
convergence predictions. The error-informed cluster points
$\theta_{WM,e}=\lambda_{71}$ and $\theta_{WM,1/\lambda}=\lambda_{93}$ provide
the fastest convergence, with $\theta_{WM,e}$ consistently giving the smallest
relative error. In contrast, $\lambda_{k+1}$, $\theta_r$, and $\theta_m$
deteriorate significantly and initially increase the relative error.
Although $\lambda_k$ is not optimal, it performs considerably better than
these convergence-oriented choices and provides a useful compromise in this
regime.

Figures~\ref{fig:Exp1Diag} and~\ref{fig:Exp1Indiv} characterize the
perturbations. The heat maps show substantial mixing within the dominant
subspace and pronounced leakage into its complement, with
$\|\widehat V_k^TQ_{k+1:n}\|=7.192\times10^{-1}$, corresponding to a largest
principal angle of approximately $45.99^\circ$. Moreover,
$\eta_i\approx\tau_i$ for most dominant eigenvectors, while $\delta_i$ is
substantially smaller, indicating that most of the perturbation energy lies
in the complement subspace. Thus, substantial leakage coincides with the
regime in which the error-analysis cluster points outperform those motivated
by exact-arithmetic convergence theory.

These results indicate that substantial leakage into the complement subspace
strongly alters the behaviour of sLMP. In this regime, the cluster points
obtained from error analysis outperform those motivated solely by
exact-arithmetic convergence theory. This observation motivates the next
experiment, in which the perturbation magnitudes are kept unchanged while the
perturbations are confined entirely to the dominant subspace.

\begin{figure}[t]
    \centering

    \begin{subfigure}[t]{0.45\textwidth}
        \centering
    \includegraphics[width=\linewidth]{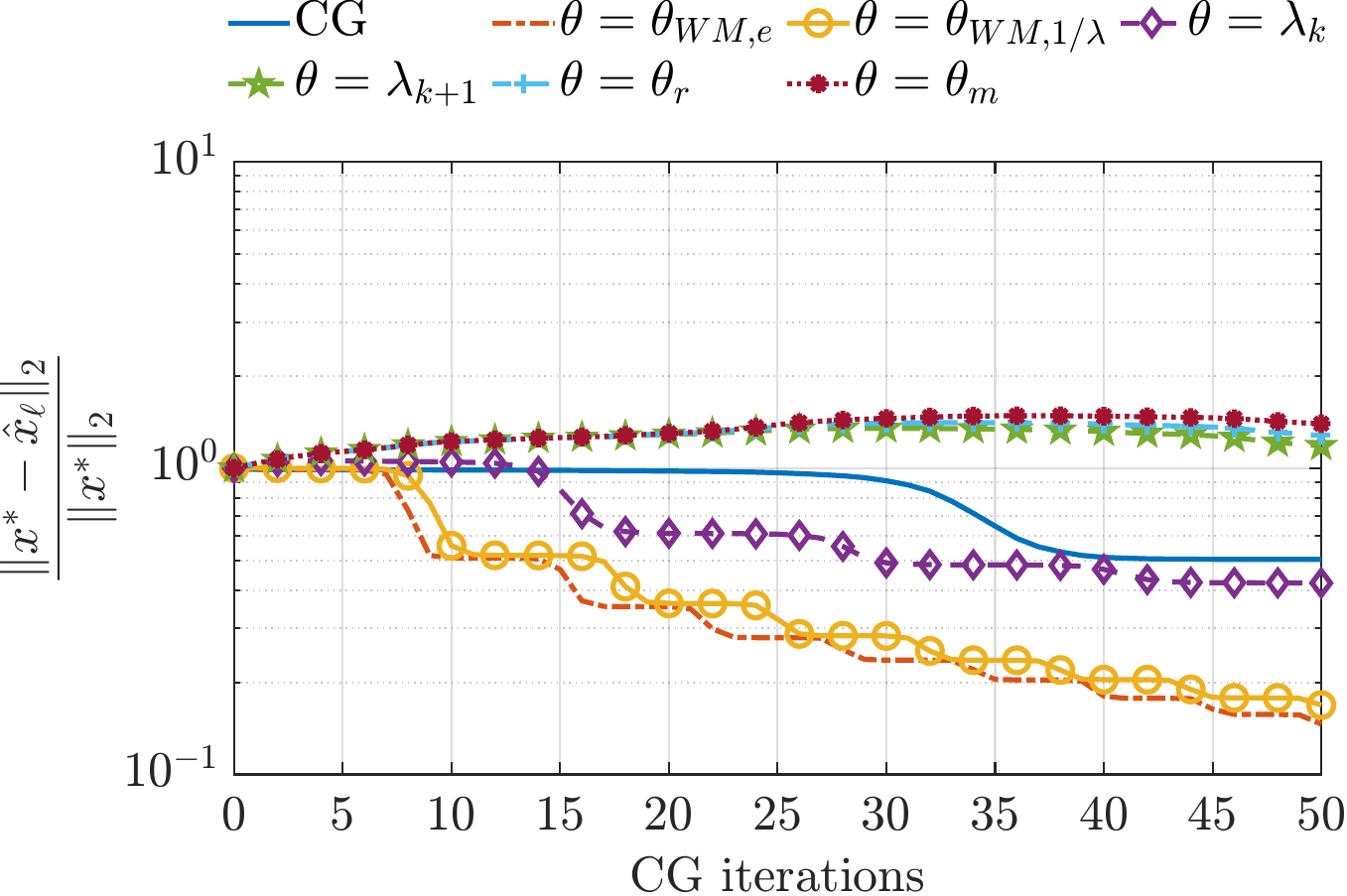}
        \caption{Relative forward-error convergence.}
        \label{fig:Exp1Conv}
    \end{subfigure}
    \begin{subfigure}[t]{0.45\textwidth}
        \centering
    \includegraphics[width=\linewidth]{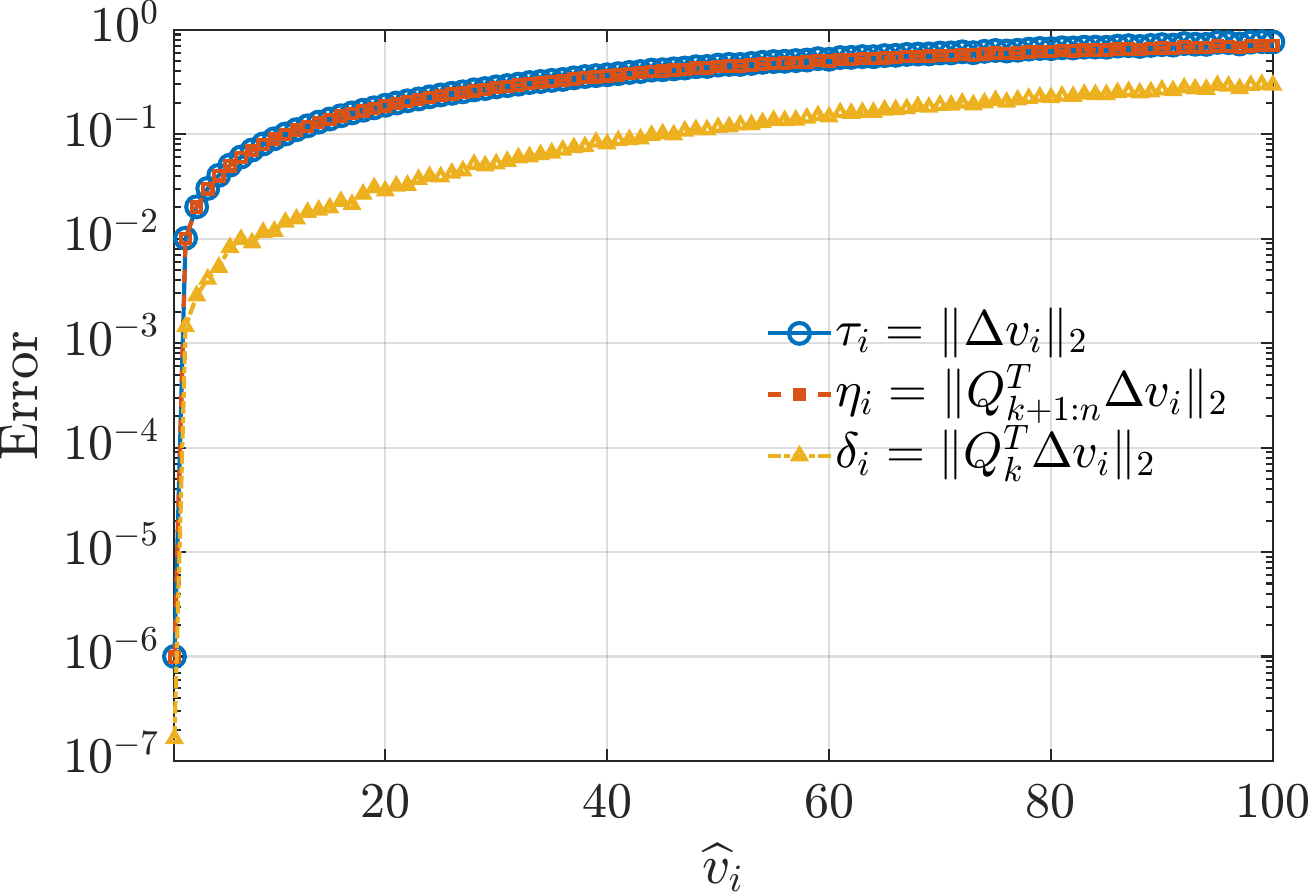}
        \caption{Individual perturbation components.}
        \label{fig:Exp1Indiv}
    \end{subfigure}

    \vspace{0.3em}

    \begin{subfigure}[t]{\textwidth}
        \centering
        \includegraphics[width=0.7\textwidth]{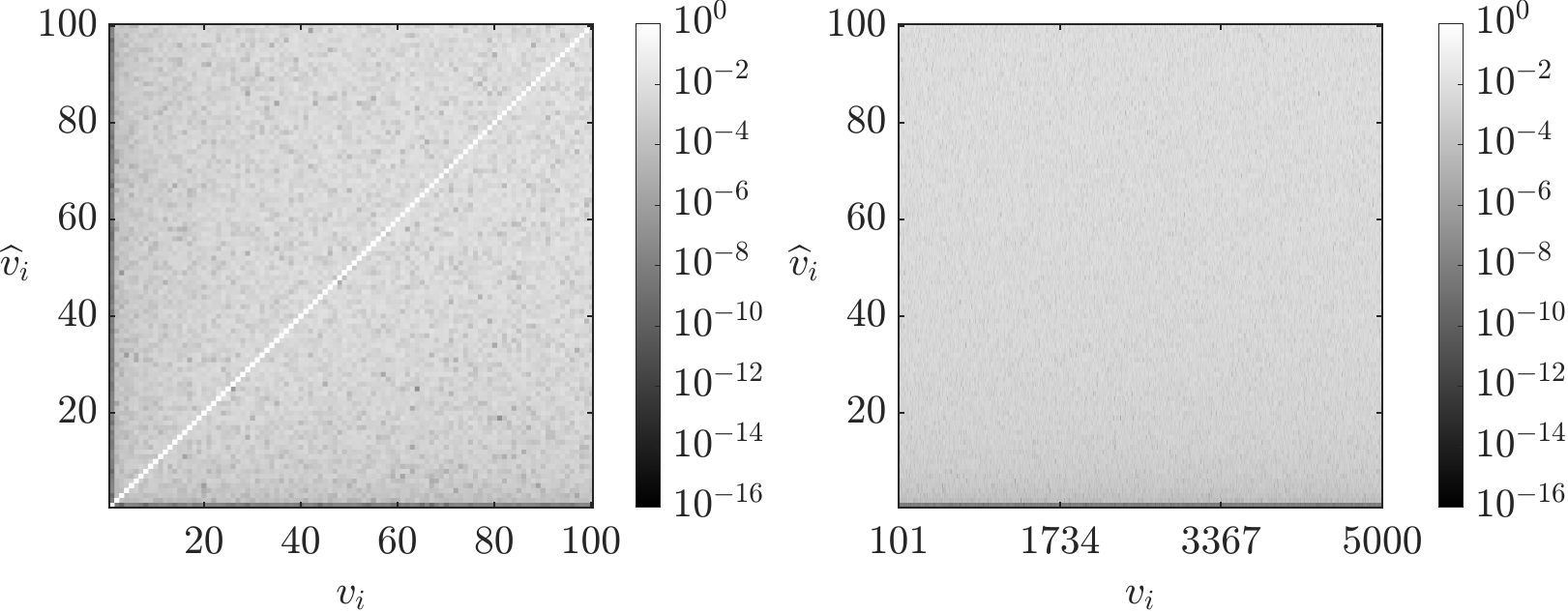}
        \caption{Perturbation diagnostics.}
        \label{fig:Exp1Diag}
    \end{subfigure}

    \caption{Results for unrestricted synthetic perturbations with
linearly distributed perturbation magnitudes
$\tau_i \in [10^{-6},1]$.
Panel (a) shows the convergence for the different cluster points;
panel (b) shows the individual perturbation measures
$\tau_i$, $\eta_i$, and $\delta_i$; and panel (c) shows the heat maps
$|\widehat{V}_k^T Q_k|$ and
$|\widehat{V}_k^T Q_{k+1:n}|$, illustrating mixing within the
dominant subspace and leakage into the complement subspace,
respectively.}
    \label{fig:Exp1}
\end{figure}

\subsubsection{Large perturbations confined to the dominant subspace}

The perturbation magnitudes are kept identical to the previous experiment,
$\tau_i\in[10^{-6},1]$, but are now confined entirely to the dominant subspace,
eliminating leakage while preserving dominant-subspace mixing.
\begin{figure}[t]
    \centering

    \begin{subfigure}[t]{0.45\textwidth}
        \centering
        \includegraphics[width=\textwidth]{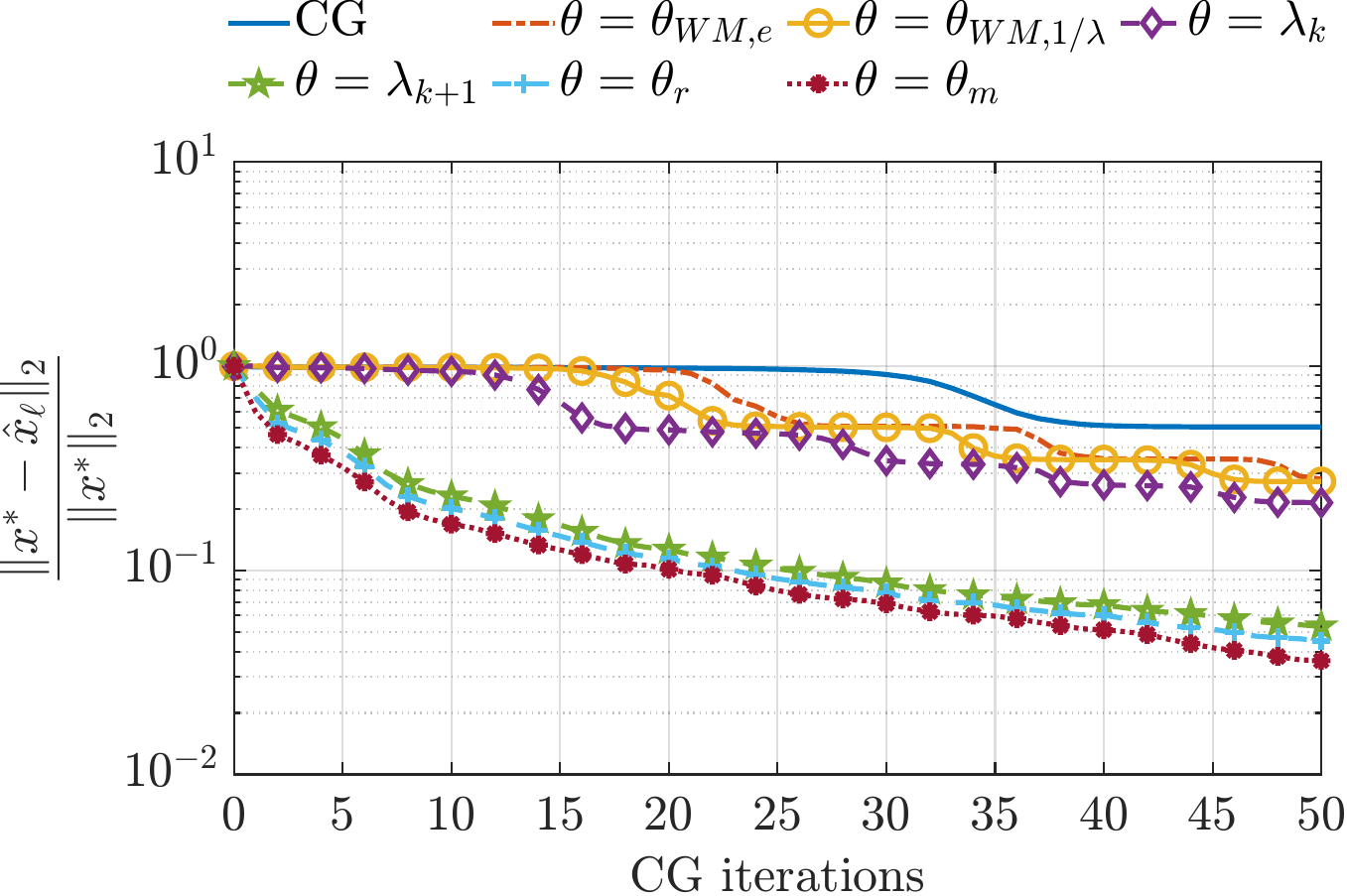}
        \caption{Relative forward-error plots.}
        \label{fig:Exp2Conv}
    \end{subfigure}
    \begin{subfigure}[t]{0.45\textwidth}
        \centering
    \includegraphics[width=\textwidth]{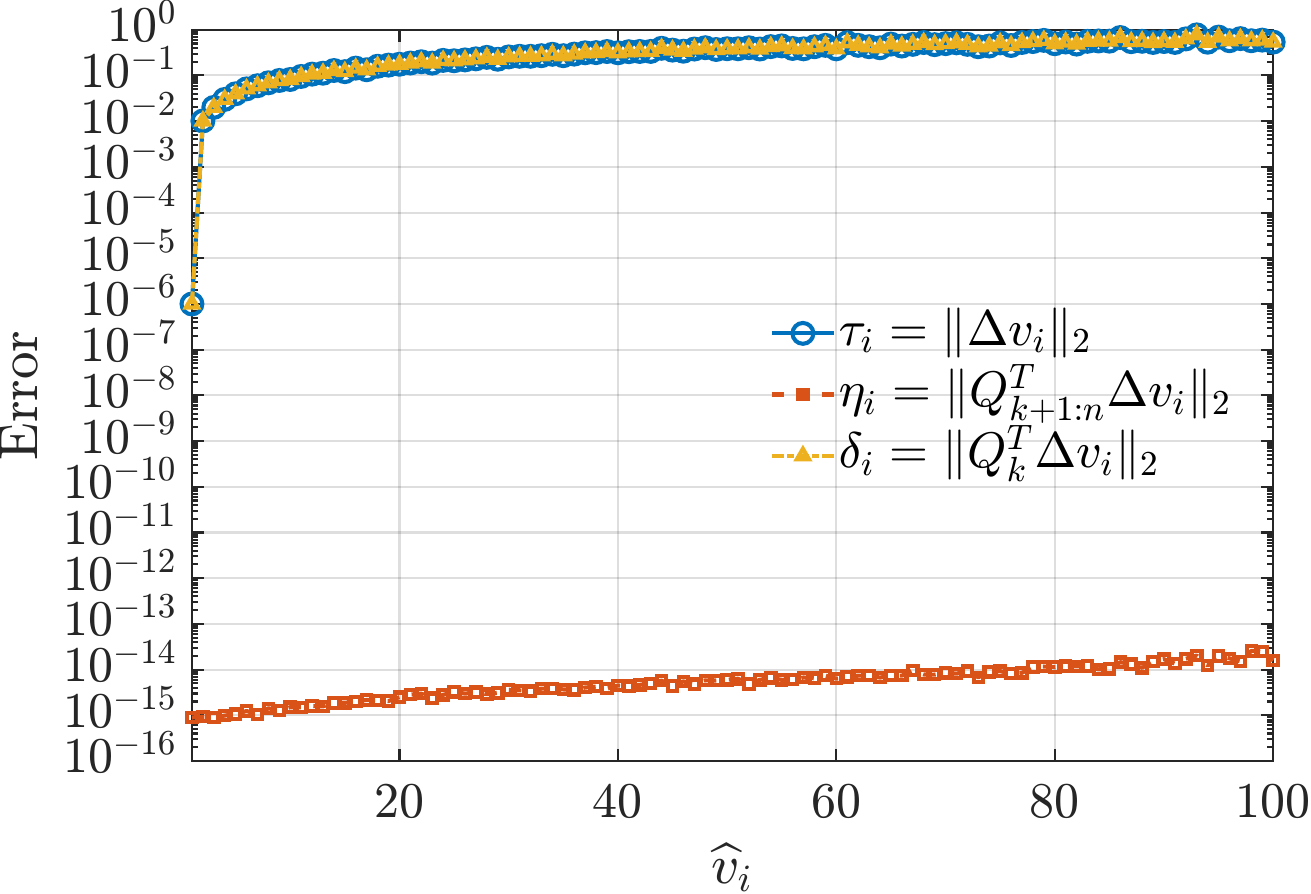}
        \caption{Individual eigenvector perturbation components.}
        \label{fig:Exp2Indiv}
    \end{subfigure}
    
    \vspace{0.3em}

    \begin{subfigure}[t]{\textwidth}
        \centering
        \includegraphics[width=0.7\textwidth]{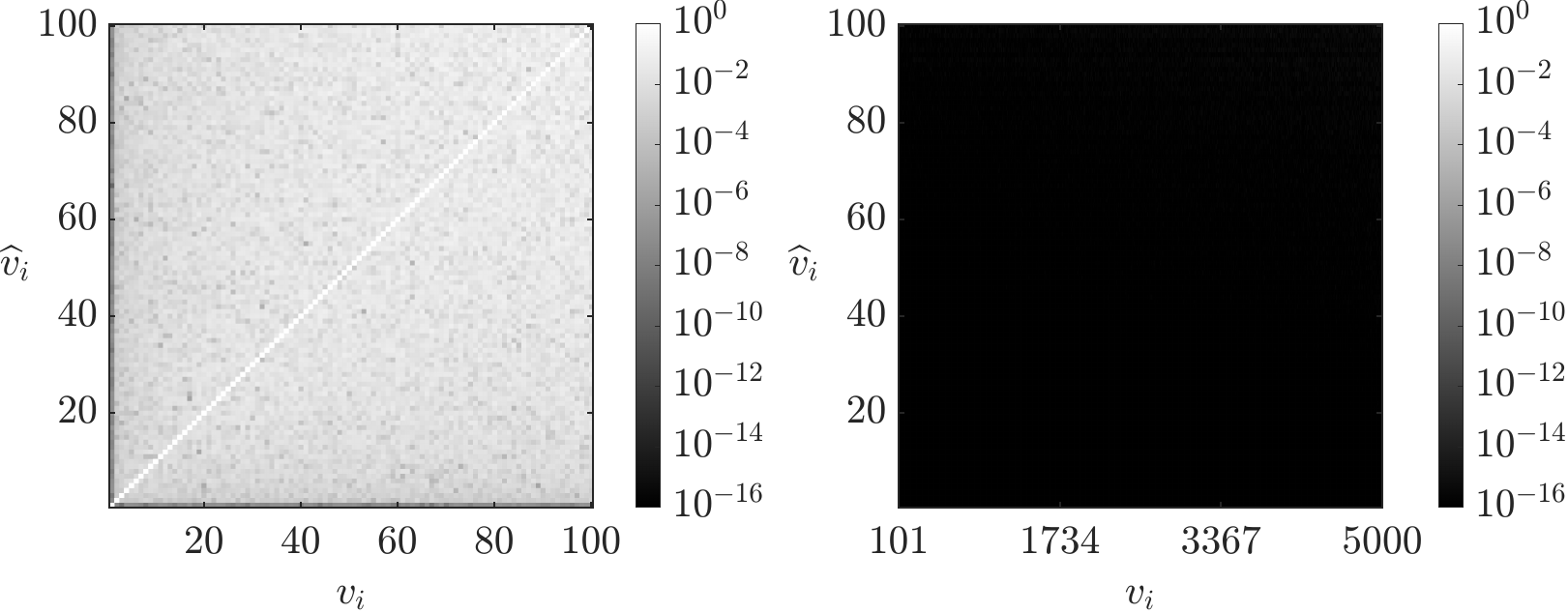}
        \caption{Perturbation diagnostics.}
        \label{fig:Exp2Diag}
    \end{subfigure}

    \caption{Results for synthetic perturbations confined to the dominant subspace with linearly distributed perturbation magnitudes
$\tau_i\in[10^{-6},1]$.
The panels are as described in
Figure~\ref{fig:Exp1}}
    \label{fig:Exp2}
\end{figure}

The convergence plots in Figure~\ref{fig:Exp2Conv} show a marked reversal
from the previous experiment. The exact-arithmetic choices now converge
fastest, with $\theta_m$ giving the smallest relative error, followed closely
by $\theta_r$ and $\lambda_{k+1}$. This is consistent with the analysis in~\cite{diouane2024efficient} for
$[\lambda_{k+1},\lambda_k]$. In contrast, $\theta_{WM,e}$ and
$\theta_{WM,1/\lambda}$ are less effective, while $\lambda_k$ lies between
the convergence-oriented and error analysis choices.

The diagnostics explain this reversal. Figure~\ref{fig:Exp2Diag} shows no
leakage beyond roundoff,
\[
\left\|\widehat V_k^TQ_{k+1:n}\right\|
=3.283\times10^{-14},
\]
while Figure~\ref{fig:Exp2Indiv} gives $\eta_i\approx0$ and
$\delta_i\approx\tau_i$. Thus, despite substantial perturbations of the individual eigenvectors, the dominant subspace itself is preserved up to roundoff. Neglecting this roundoff-level leakage, there exists an orthogonal matrix
$R\in\mathbb{R}^{k\times k}$ such that $\widehat V_k=Q_kR$, and
\[
H_kA
=
\theta Q_kRD_k^{-1}R^TD_kQ_k^T
+
Q_{k+1:n}D_{k+1:n}Q_{k+1:n}^T,
\]
where $D_k=\operatorname{diag}(\lambda_1,\ldots,\lambda_k)$. Hence, the
complement eigenvalues remain exactly $\lambda_{k+1},\ldots,\lambda_n$.

Let $M=RD_k^{-1}R^TD_k$. It is similar to the symmetric positive definite
matrix
$\widetilde M=D_k^{1/2}RD_k^{-1}R^TD_k^{1/2}$, and hence its eigenvalues
$\mu_1,\ldots,\mu_k$ are real and positive. Moreover,
$\det(M)=1$, so $\prod_{i=1}^k\mu_i=1$. The dominant eigenvalues of the
preconditioned matrix are therefore $\theta\mu_1,\ldots,\theta\mu_k$, with
geometric mean
\[
\left(\prod_{i=1}^{k}\theta\mu_i\right)^{1/k}=\theta.
\]
Thus, the complement spectrum is preserved while the clustered dominant
spectrum scales linearly with $\theta$. In particular,
$\lambda_{\max}(\theta M)=\theta\lambda_{\max}(M)$. Since for this experiment
$\theta_m<\theta_r<\lambda_{k+1}$,
\[
\lambda_{\max}(\theta_mM)
<
\lambda_{\max}(\theta_rM)
<
\lambda_{\max}(\lambda_{k+1}M),
\]
consistent with Figure~\ref{fig:Exp2}(a), where $\theta_m$ converges fastest,
followed by $\theta_r$ and $\lambda_{k+1}$. Comparing
Figures~\ref{fig:Exp1} and~\ref{fig:Exp2} therefore indicates that the
deterioration in the previous experiment is caused primarily by leakage into
the complement subspace rather than by dominant-subspace mixing alone.

\subsubsection{Small perturbations with unrestricted directions}

We return to unrestricted perturbations, but reduce their magnitudes to
$\tau_i\in[10^{-12},10^{-6}]$. This tests whether sufficiently accurate
spectral information restores the effectiveness of the convergence-oriented
cluster points despite the presence of leakage.

\begin{figure}[t]
    \centering
    \begin{subfigure}[t]{0.45\textwidth}
        \centering
        \includegraphics[width=\textwidth]{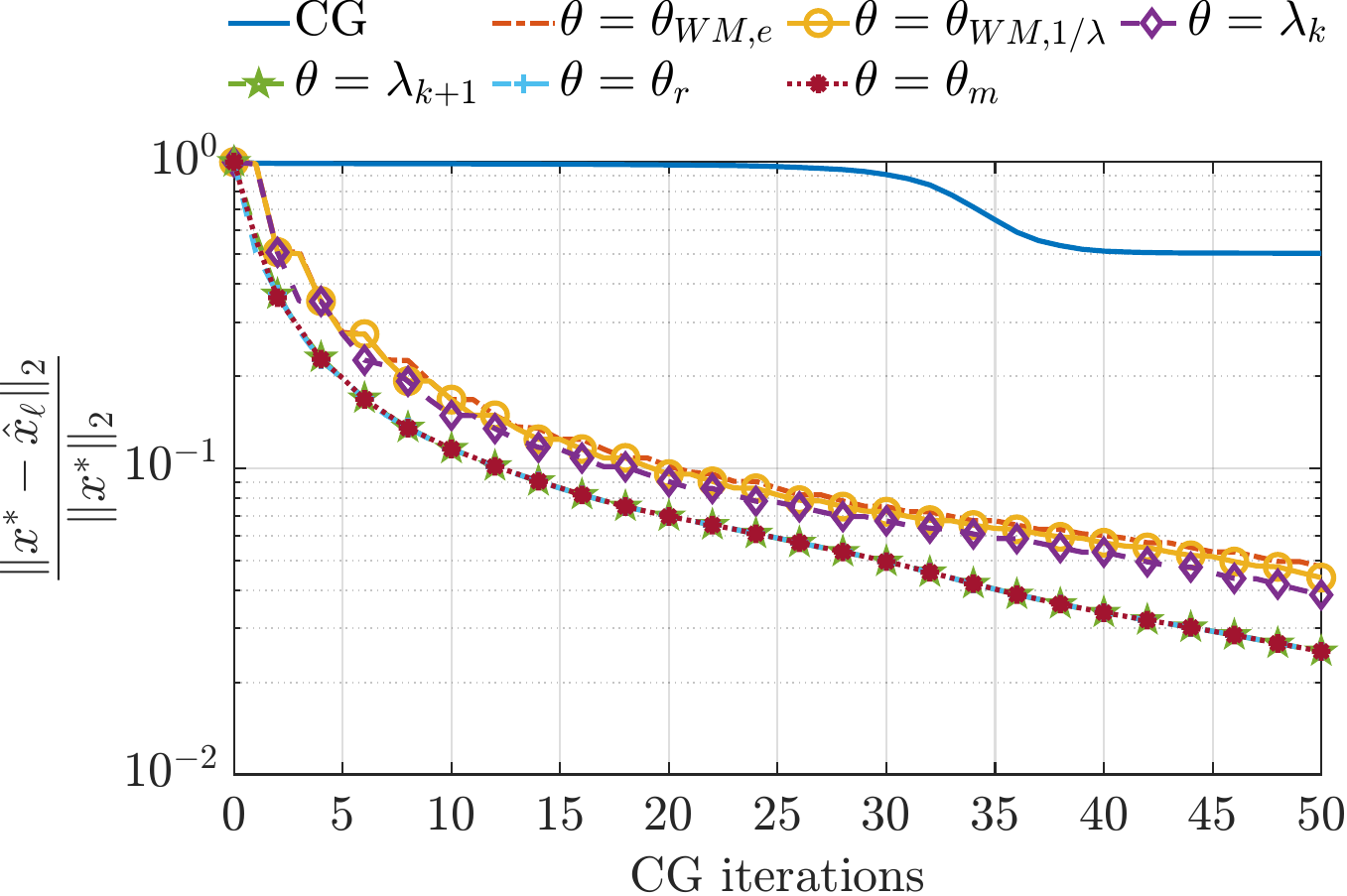}
        \caption{Relative forward-error convergence.}
    \end{subfigure}
     \begin{subfigure}[t]{0.45\textwidth}
        \centering
    \includegraphics[width=\textwidth]{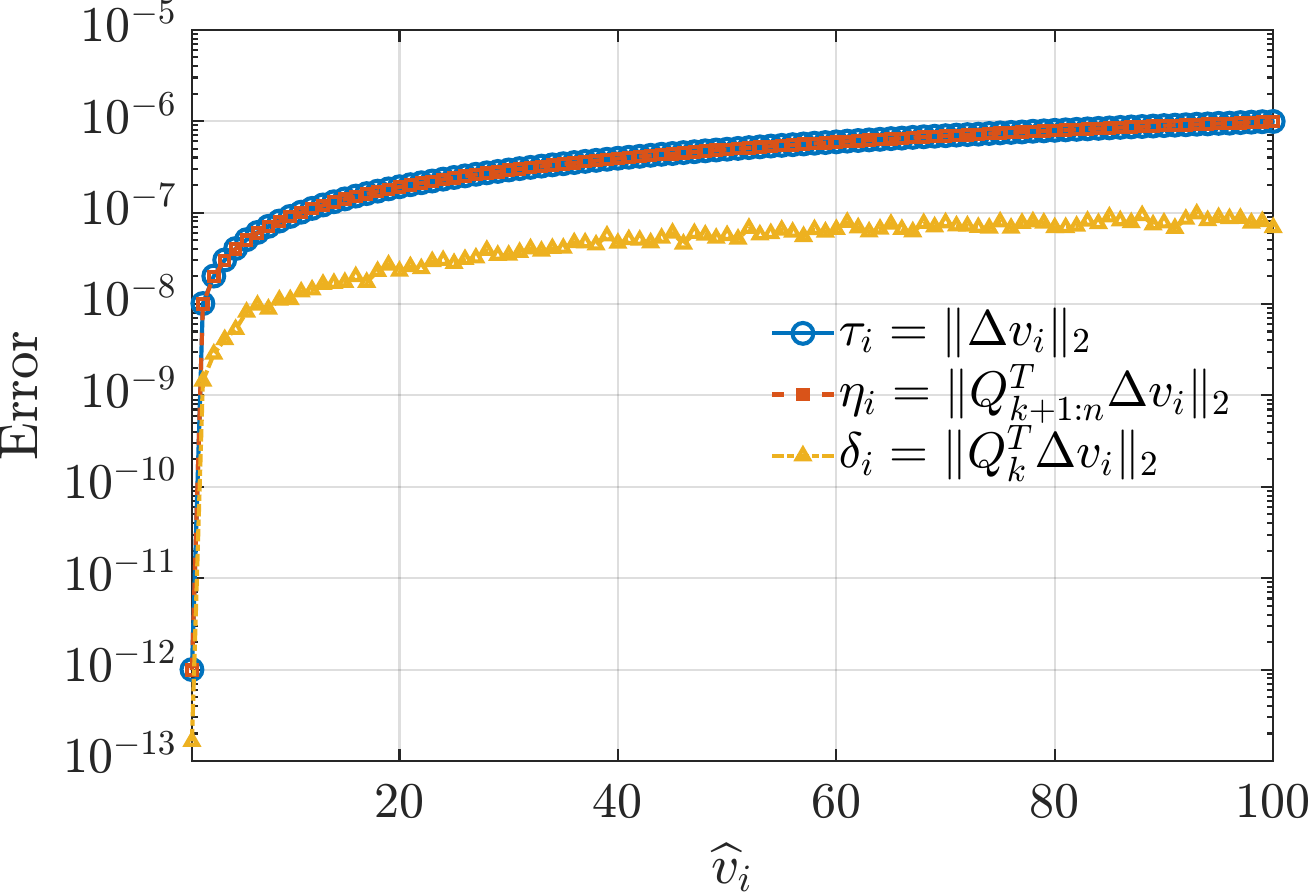}
        \caption{Individual eigenvector perturbation components.}
    \end{subfigure}
    
    \vspace{0.5em}

    \begin{subfigure}[t]{\textwidth}
        \centering
        \includegraphics[width=0.7\textwidth]{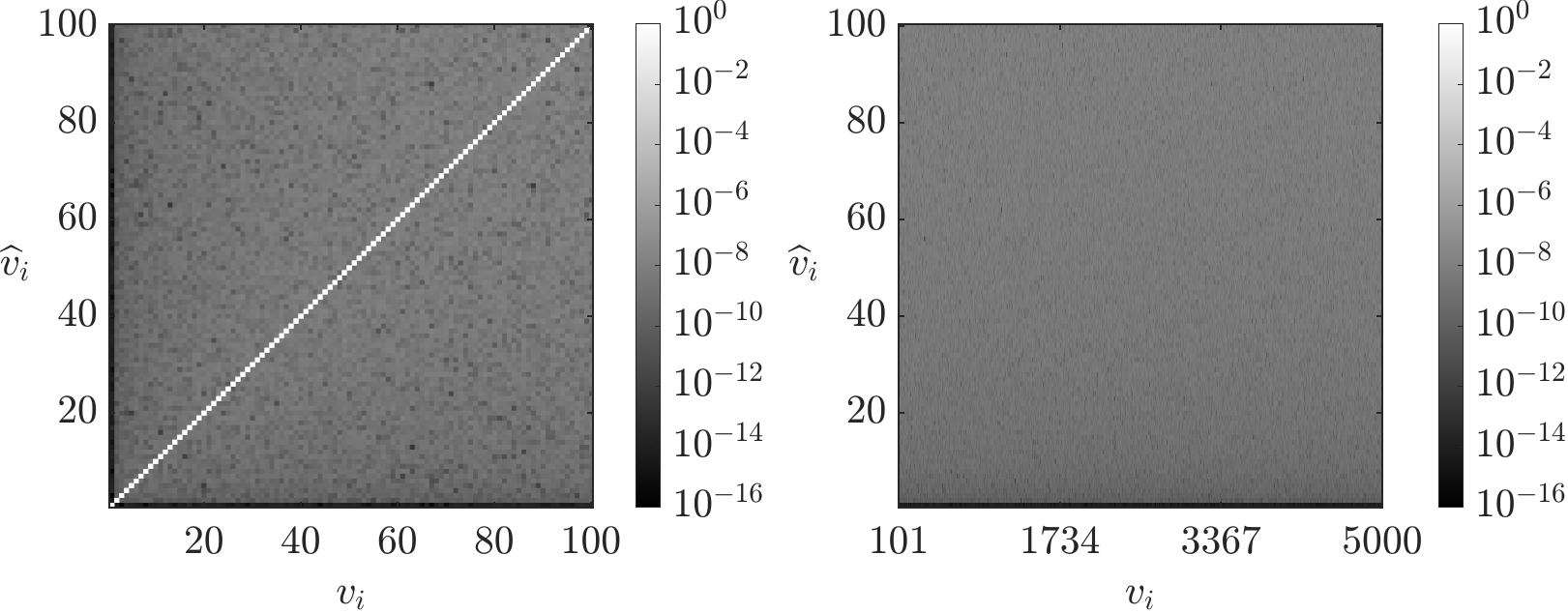}
        \caption{Perturbation diagnostics.}
    \end{subfigure}
    
   \caption{Results for unrestricted synthetic perturbations with logarithmically
distributed perturbation magnitudes
$\tau_i\in[10^{-12},10^{-6}]$.
The panels are as described in
Figure~\ref{fig:Exp1}}
    \label{fig:Exp3}
\end{figure}

Figure~\ref{fig:Exp3}(a) shows that reducing the perturbation magnitude
almost completely restores the exact-arithmetic convergence behaviour:
the convergence-oriented cluster points are again fastest, while
$\theta_{WM,e}$ and $\theta_{WM,1/\lambda}$ no longer have the advantage
observed in the first experiment.

The diagnostics in Figure~\ref{fig:Exp3}(b) show that leakage remains present
but is several orders of magnitude smaller, with $\left\|\widehat V_k^TQ_{k+1:n}\right\|
=1.003\times10^{-6}.$
Both $\delta_i$ and $\eta_i$ decrease with the perturbation magnitude, so the
remaining leakage is insufficient to significantly modify the dominant
subspace.

Comparing Figures~\ref{fig:Exp1}, \ref{fig:Exp2}, and \ref{fig:Exp3} shows
that both the direction and magnitude of the perturbations govern the
practical behaviour of sLMP. Large leakage favours the error analysis
cluster points, whereas eliminating or sufficiently reducing leakage restores
the superiority of the convergence-oriented choices. Thus, both perturbation
direction and magnitude determine whether exact-arithmetic convergence theory
remains predictive in finite precision.

\subsection{Practical eigensolvers}
\label{sec:practical-eigensolvers}
We now investigate whether the mechanisms observed in
Subsection~\ref{Result sec: Synthetic perturbation} also arise when the
dominant spectral information is obtained from practical eigensolvers.
Both the approximate dominant eigenvectors and eigenvalues are used to
construct sLMP, while the exact quantities $\lambda_{k+1}$,
$\theta_m$, and $\theta_r$ are retained as reference cluster points for
comparison with exact-arithmetic convergence predictions.

\subsubsection{Randomized Nystr\"om Algorithm}
In this experiment, we construct sLMP using approximate dominant
eigenpairs computed by the randomized Nyström algorithm~\cite{halko2011finding},
following Algorithm~4 in~\cite{dauvzickaite2021randomised}.
\begin{figure}[t]
    \centering

    \begin{subfigure}[t]{0.45\textwidth}
        \centering
       \includegraphics[width=\textwidth]
        {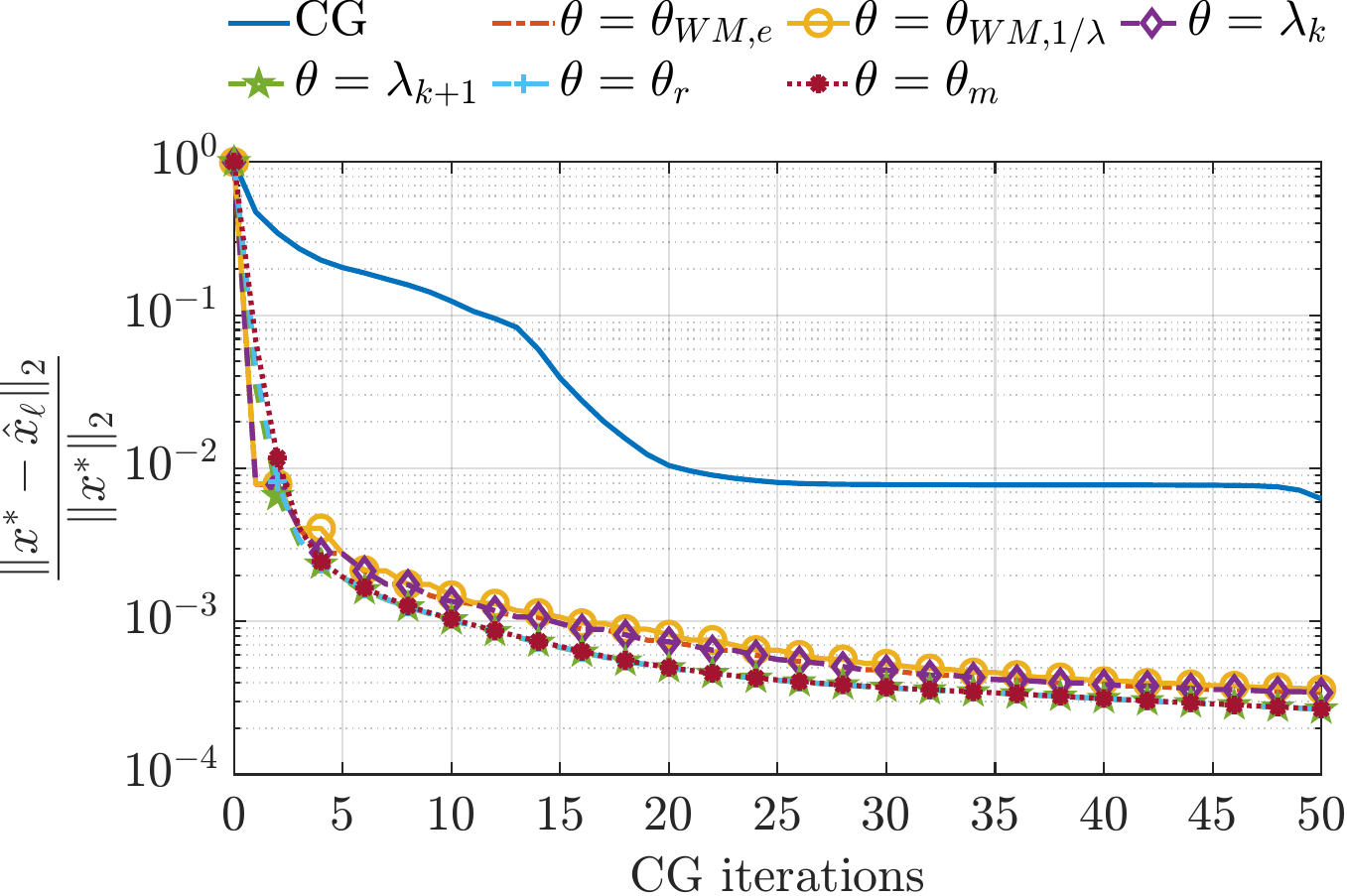}
        \caption{Relative forward-error plots}
        \label{fig:NystromConv}
    \end{subfigure}
       \begin{subfigure}[t]{0.45\textwidth}
        \centering
    \includegraphics[width=\textwidth]
        {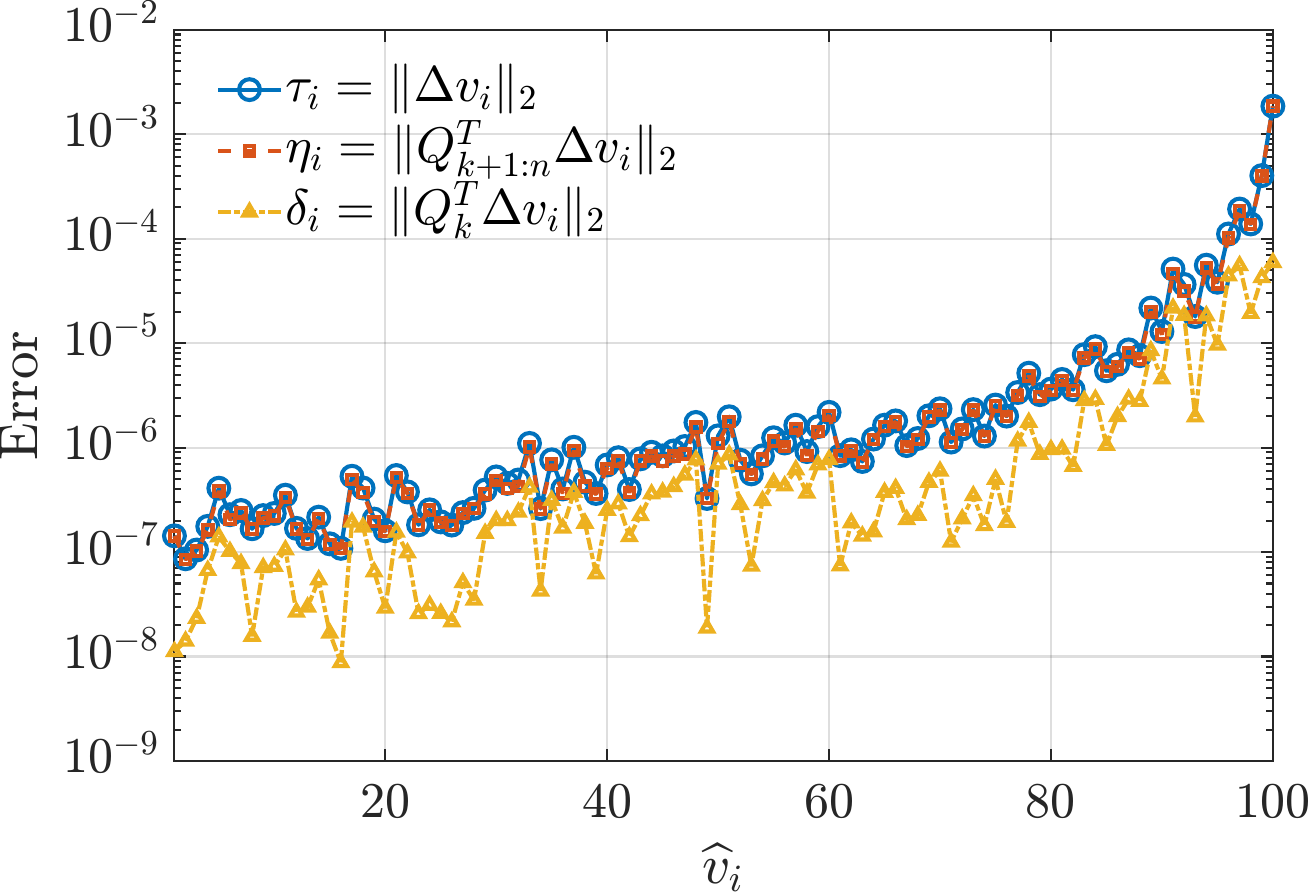}
        \caption{Individual eigenvector perturbation components}
        \label{fig:NystromIndiv}
    \end{subfigure}
    \vspace{0.5em}

    \begin{subfigure}[t]{\textwidth}
        \centering
        \includegraphics[width=0.7\textwidth]
        {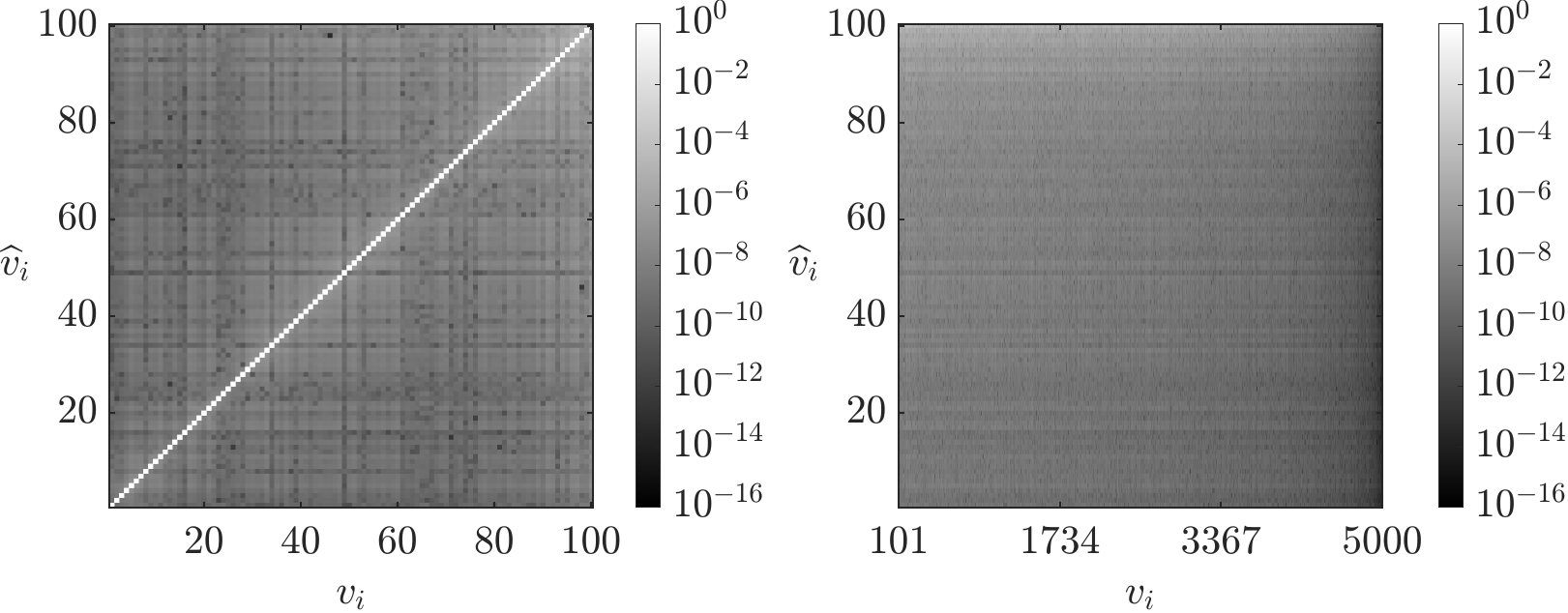}
        \caption{Perturbation diagnostics}
        \label{fig:NystromDiag}
    \end{subfigure}

   \caption{Results obtained using approximate dominant eigenpairs computed
by the randomized Nystr\"om algorithm. The panels are as described in
Figure~\ref{fig:Exp1}}
    \label{fig:Nystrom}
\end{figure}
Figure~\ref{fig:NystromConv} shows that the convergence-oriented cluster
points remain the most effective: $\theta_m$, $\theta_r$, and
$\lambda_{k+1}$ give the fastest and nearly coincident convergence.
In contrast, $\theta_{WM,1/\lambda}$ and $\widehat\lambda_k$ converge
more slowly, with the latter coinciding with the weighted-median curve
because $\theta_{WM,e}=\widehat\lambda_k$ in this experiment.

The diagnostics in Figures~\ref{fig:NystromDiag}
and~\ref{fig:NystromIndiv} show that most computed eigenvectors remain
accurately aligned with their exact counterparts, while both mixing and
leakage increase near the end of the dominant spectrum. The global
leakage measure is
$\left\|\widehat V_k^TQ_{k+1:n}\right\|=1.863\times10^{-3}$,
but the leakage is strongly localized, with the final approximate
eigenvector providing the dominant contribution. In particular,
$\eta_k$ accounts for most of its perturbation, while $\delta_k$ is
smaller. Consequently, the perturbation weight associated with
$\widehat v_k$ dominates the weighted-median criterion, giving
$\theta_{WM,e}=\widehat\lambda_k$.

This choice has a useful interpretation. Since
$\widehat\alpha_i=1-\frac{\theta}{\widehat\lambda_i}$, setting
$\theta=\widehat\lambda_k$ gives $\widehat\alpha_k=0$, so that
$I-\widehat\alpha_k\widehat v_k\widehat v_k^T=I$ and the least accurate
eigenvector does not contribute to sLMP. Thus, the weighted-median
criterion automatically suppresses the dominant eigenvector carrying
the largest perturbation. Nevertheless, since the remaining computed
eigenvectors are highly accurate, the overall perturbation remains
limited and the convergence-oriented choices retain their advantage.
This is consistent with the synthetic experiments: leakage can alter
the preferred cluster point when it is sufficiently large and
distributed across the approximate dominant basis, whereas an isolated
large error can be mitigated by removing the corresponding inaccurate
vector from the preconditioner.

\subsubsection{REVD-ritzit}

We next consider the REVD-ritzit method proposed
in~\cite{dauvzickaite2021randomised}. It approximates the dominant
eigenpairs using a single, parallelizable matrix-matrix product and was
found to provide the best overall preconditioning performance among the
randomized methods considered in~\cite{dauvzickaite2021randomised}.

\begin{figure}[t]
    \centering

    \begin{subfigure}[t]{0.45\textwidth}
        \centering
        \includegraphics[width=\textwidth]{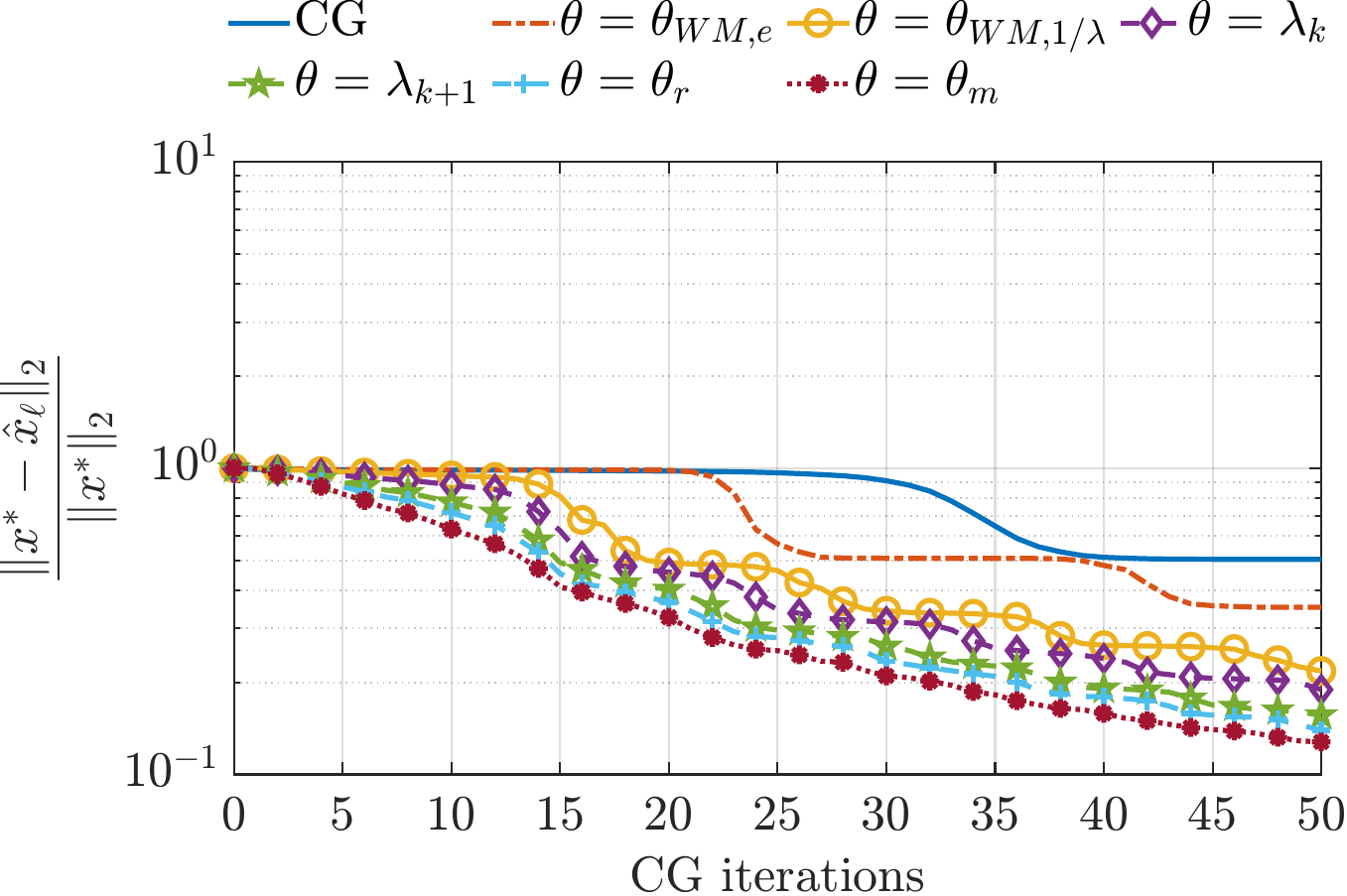}
        \caption{Relative forward-error plots}
        \label{fig:RitzitConv}
    \end{subfigure}
    \begin{subfigure}[t]{0.45\textwidth}
        \centering
        \includegraphics[width=\textwidth]{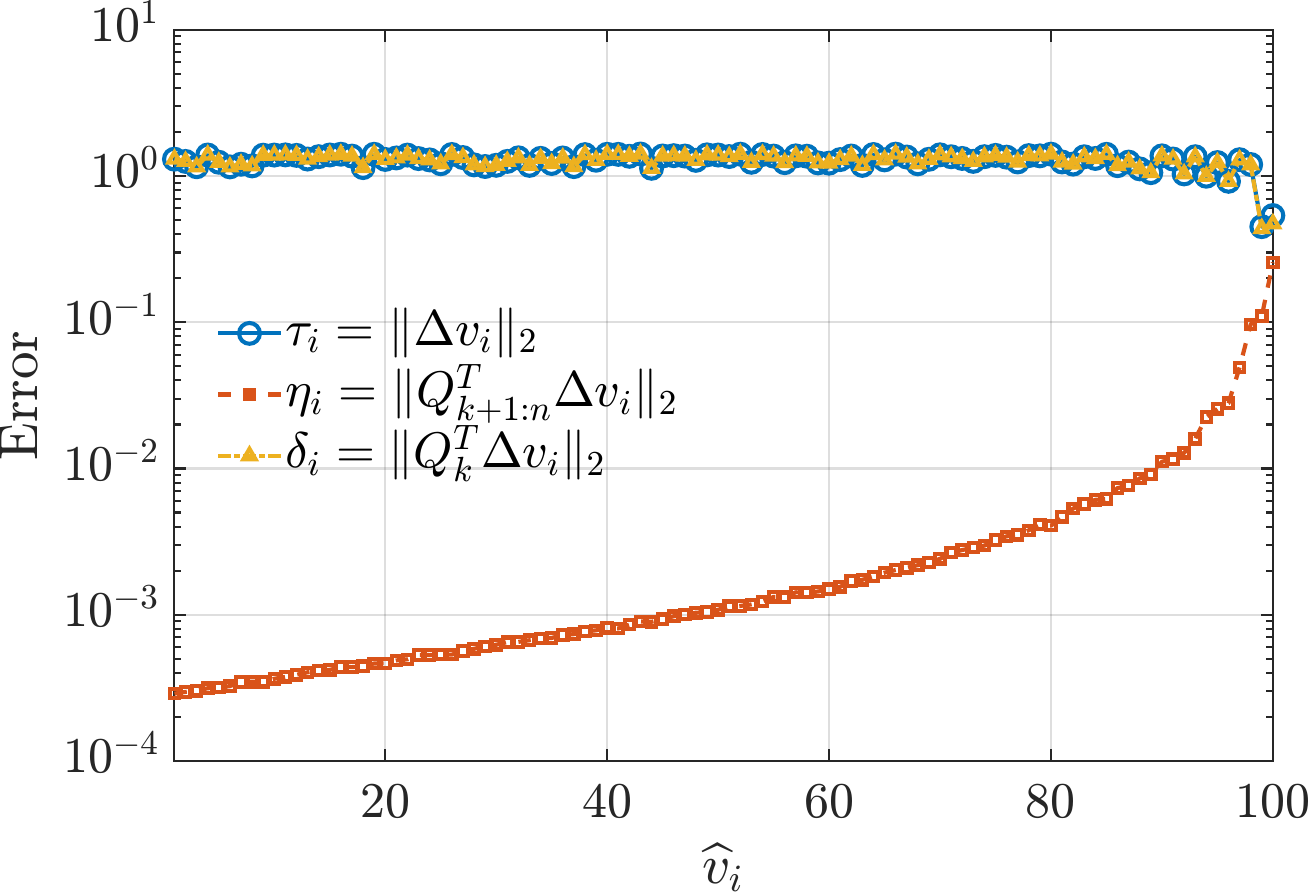}
        \caption{Perturbation diagnostics}
        \label{fig:RitzitIndiv}
    \end{subfigure}
    \vspace{0.5em}

    \begin{subfigure}[t]{\textwidth}
        \centering
        \includegraphics[width=0.7\textwidth]{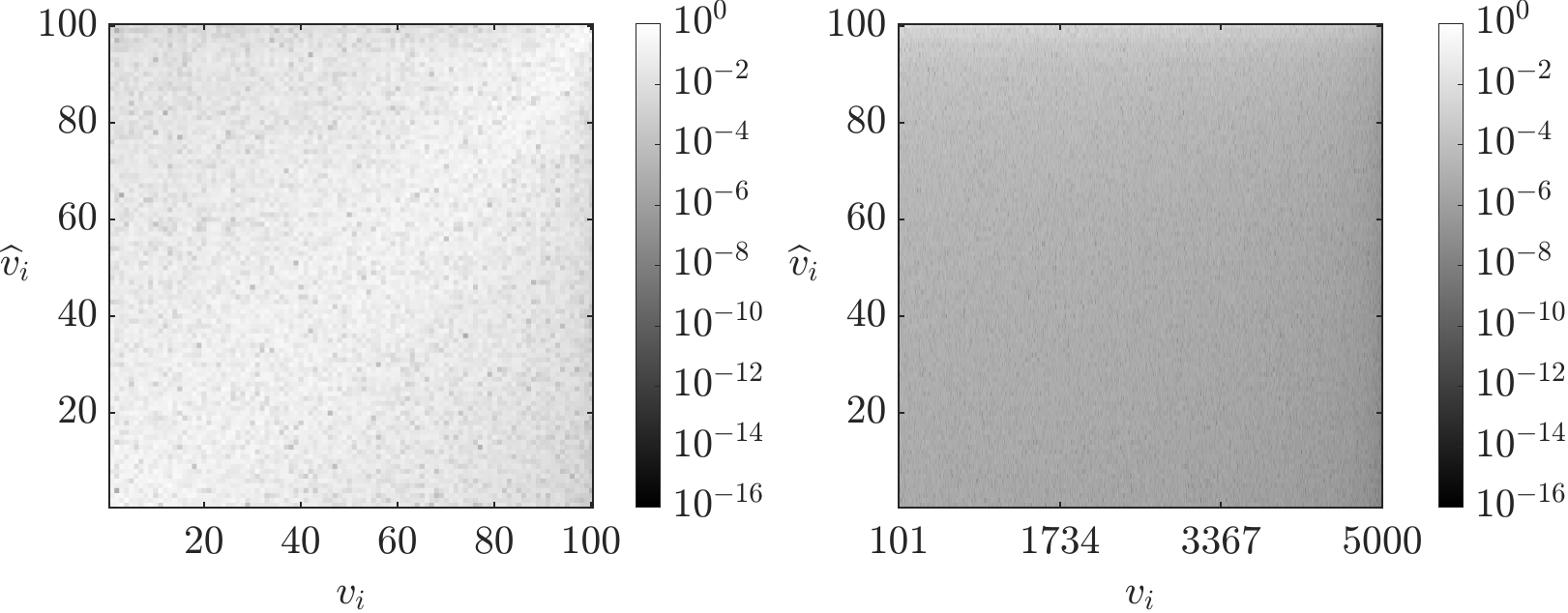}
        \caption{Perturbation diagnostics}
        \label{fig:RitzitDiag}
    \end{subfigure}
    \caption{Results obtained using approximate dominant eigenpairs computed
by the randomized REVD-ritzit eigensolver. The panels are as described in
Figure~\ref{fig:Exp1}}
    \label{fig:Ritzit}
\end{figure}

Figure~\ref{fig:RitzitConv} shows that the convergence-oriented cluster
points again perform best. The choices $\theta_m$, $\theta_r$, and
$\lambda_{k+1}$ yield the smallest relative errors, whereas
$\theta_{WM,1/\lambda}$, $\theta_{WM,e}$, and $\widehat\lambda_k$
converge more slowly. This resembles the second synthetic perturbation
experiment, where the perturbations were confined to the dominant
subspace.

The diagnostics in Figures~\ref{fig:RitzitIndiv},~\ref{fig:RitzitDiag} show substantial mixing
among the dominant eigenvectors but comparatively little leakage into
the complement subspace. In particular, $\tau_i\approx\delta_i$ for
almost all dominant eigenvectors, whereas $\eta_i$ remains several
orders of magnitude smaller throughout most of the dominant spectrum
and increases only for the final few eigenvectors. Thus, despite large
individual eigenvector perturbations, the perturbation is almost
entirely confined to the dominant subspace.

Since the dominant subspace is largely preserved, the complement
spectrum is only weakly affected and the convergence-oriented cluster
points recover their expected superiority, consistent with the second
synthetic perturbation experiment. Comparing Figures~\ref{fig:Nystrom}
and~\ref{fig:Ritzit} shows that the two eigensolvers produce different
perturbation structures: Nystr\"om concentrates most of the error in the
final dominant eigenvectors, whereas REVD-ritzit distributes it more
uniformly while largely preserving the dominant subspace. Nevertheless,
both experiments support the synthetic perturbation results: despite
the additional eigenvalue inaccuracies, preservation of the dominant
subspace remains the primary factor determining whether the
convergence-oriented cluster points remain effective.

\section{Conclusion}
\label{sec: conclusion}
In this work, we have investigated the choice of cluster point in sLMP
from a finite-precision perspective. Our analysis shows that choices favourable
for exact-arithmetic convergence need not provide the best practical behaviour
when rounding errors and inaccuracies in the spectral information are taken
into account. For rounding errors arising during the application of sLMP,
we derived relative for-ward-error bounds in the dominant and complement
subspaces. In particular, the dominant-subspace analysis identifies the
weighted median $\theta_{WM,1/\lambda}$ as a finite-preci-sion choice and
explains the poor practical behaviour of the standard choice $\theta=1$.

For inaccuracies in the spectral information, we derived a perturbation bound
accounting for errors in both the dominant eigenvectors and eigenvalues. Its
minimizer $\theta_{WM,e}$ is a weighted median of the perturbed dominant
eigenvalues, with weights determined by the eigenvector perturbation
magnitudes, and reduces to a weighted median of the exact dominant eigenvalues
when the eigenvalues are exact. Since these perturbation magnitudes are
generally unavailable in practice, $\theta_{WM,e}$ primarily serves as a
theoretically optimal reference.

The numerical experiments show that the effect of inaccurate spectral
information depends on both the magnitude and direction of the eigenvector
errors. Significant leakage into the complement subspace can favour
$\theta_{WM,1/\lambda}$ and $\theta_{WM,e}$ over exact-arithmetic
convergence choices, whereas preserving the dominant invariant subspace
restores the advantage of the convergence-oriented choices. Practical
eigensolvers exhibit the same qualitative behaviour, with eigenvalue
inaccuracies having a comparatively small additional effect in the cases
considered. The experiments also indicate that $\lambda_k$, while not
generally optimal, can provide a useful compromise when other
convergence-oriented choices deteriorate.

Overall, the choice of cluster point reflects a balance between convergence
and numerical stability, governed by the quality and structure of the
available spectral information. In particular, preservation of the dominant
subspace provides an important indicator of when exact-arithmetic convergence
theory remains predictive of the practical behaviour of sLMP. Future work could investigate sLMP in lower-precision arithmetic, where rounding-error effects may play an even greater role in the choice of cluster point.
\bibliographystyle{siamplain}
\bibliography{references}
\end{document}

%% file: ex_shared.tex
\usepackage{lipsum}
\usepackage{amsfonts}
\usepackage{graphicx}
\usepackage{epstopdf}
\usepackage{algorithmic}
\usepackage{subcaption}
\ifpdf
  \DeclareGraphicsExtensions{.eps,.pdf,.png,.jpg}
\else
  \DeclareGraphicsExtensions{.eps}
\fi

\newsiamremark{remark}{Remark}
\newsiamremark{hypothesis}{Hypothesis}
\crefname{hypothesis}{Hypothesis}{Hypotheses}
\newsiamthm{claim}{Claim}
\newsiamremark{fact}{Fact}
\crefname{fact}{Fact}{Facts}

\headers{Cluster Point Effects on Rounding and Sensitivity}{H. Elzayyadi and J. M. Tabeart}

\title{The Influence of the Cluster Point on Rounding Errors and Sensitivity in the Spectral Limited-Memory Preconditioner\thanks{Submitted to the editors September 22, 2026.
\funding{This work was funded by the Irène Curie fellowship of Eindhoven University of Technology}}}

\author{
Hisham Elzayyadi\thanks{
Department of Mathematics and Computer Science, Eindhoven University of Technology
(\email{h.g.m.h.elzayyadi@tue.nl}).
}
\and
Jemima M. Tabeart\thanks{
Department of Mathematics and Computer Science, Eindhoven University of Technology
(\email{j.m.tabeart@tue.nl}).
}
}

\usepackage{amsopn}
